\documentclass[onecolumn,reqno,11pt]{amsart}
\usepackage{amsthm}
\usepackage{times,amssymb,amsmath,amsfonts}
\usepackage{graphicx,color}

\newtheorem{theorem}{Theorem}[section]
\newtheorem{lemma}[theorem]{Lemma}

\theoremstyle{definition}
\newtheorem{definition}[theorem]{Definition}
\newtheorem{example}[theorem]{Example}
\newtheorem{remark}[theorem]{Remark}

\date{\today}

\begin{document}
\title[Energy bounds for weighted spherical codes and designs]{Energy bounds for weighted spherical codes and designs \\ via linear programming }

\author[S. Borodachov]{S. V. Borodachov}
\address{Department of Mathematics, Towson University, 7800 York Rd, Towson, MD, 21252, USA}
\email{sborodachov@towson.edu}

\author[P. Boyvalenkov]{P. G. Boyvalenkov}
\address{ Institute of Mathematics and Informatics, Bulgarian Academy of Sciences,
8 G Bonchev Str., 1113  Sofia, Bulgaria}
\email{peter@math.bas.bg}

\author[P. Dragnev]{P. D. Dragnev}
\address{ Department of Mathematical Sciences, Purdue University \\
Fort Wayne, IN 46805, USA }
\email{dragnevp@pfw.edu}

\author[D. Hardin]{D. P. Hardin}
\address{ Center for Constructive Approximation, Department of Mathematics \\
Vanderbilt University, Nashville, TN 37240, USA }
\email{doug.hardin@vanderbilt.edu}

\author[E. Saff]{E. B. Saff}
\address{ Center for Constructive Approximation, Department of Mathematics \\
Vanderbilt University, Nashville, TN 37240, USA }
\email{edward.b.saff@vanderbilt.edu}

\author[M. Stoyanova]{M. M. Stoyanova}
\address{ Faculty of Mathematics and Informatics, Sofia University ``St. Kliment Ohridski"\\
5 James Bourchier Blvd., 1164 Sofia, Bulgaria}
\email{stoyanova@fmi.uni-sofia.bg}

	 \begin{abstract}  
Universal bounds for the potential energy of weighted spherical codes are obtained by linear programming. The universality is in the sense of Cohn-Kumar -- every attaining code is optimal with respect to a large class of potential functions (absolutely monotone), in the sense of Levenshtein -- there is a bound for every weighted code, and in the sense of parameters (nodes and weights) -- they are independent of the potential function. We derive a necessary 
condition for optimality (in the linear programming framework) of our lower bounds which is also shown to be sufficient when the potential is 
strictly absolutely monotone. Bounds are also obtained for the weighted energy of weighted spherical designs. We explore our bounds for 
several previously studied weighted spherical codes. 
\end{abstract}

 		\maketitle


\section{Introduction}

Let $\mathbb{S}^{n-1}$ be the unit sphere in $\mathbb{R}^n$. A pair $(C,W)$ consisting of an $N$-tuple 
$C=(x_1,x_2,\ldots,x_N)$ of distinct points on $ \mathbb{S}^{n-1}$ 
and corresponding 
weights $W=(w_1,w_2,\ldots,w_N)$, where $w_i>0$ corresponds to $x_i$ and $w_1+w_2+\cdots+w_N=1$, 
is called {\em a weighted spherical code}. Let $s(C):=\max \{ x_i \cdot x_j : i \neq j \}$ be the maximal inner product of 
pairs of distinct points of $C$ (or, for short, the maximal inner product of $C$). We shall denote by $|X|$ the size of a tuple $X$.

For a given (extended real-valued) continuous function $h:[-1,1] \to \mathbb{R} \cup \{+\infty\}$, finite on $[-1,1)$, and 
a given weighted spherical code, we consider the {\em weighted $h$-energy of $(C,W)$}
\[ E_h(C,W):=\sum_{i \neq j}w_iw_jh(x_i \cdot x_j),\]
which arises in the electrostatic problem of distributing $N=|W|$ positive charges (not necessa\-ri\-ly equal) on the unit sphere. 
Let 
\begin{equation} \label{min-energy}
\mathcal{E}^h(W):=\inf \{ E_h(C,W)\, :\, C\in \left( \mathbb{S}^{n-1}\right)^{|W|} \}
\end{equation}
be the {\em minimum weighted $h$-energy} among all weighted codes $(C,W)$ with fixed weight $W$.

Similarly, for fixed $s \in [-1,1)$ we consider 
\begin{equation} \label{min-w}
\mathcal{U}^h(s,W):=\sup \{ E_h(C,W):C\in \left( \mathbb{S}^{n-1}\right)^{|W|} , \, s(C) = s\},  
\end{equation}
the {\em maximum\footnote{By convention, we define the $\sup$ of the empty set to be $-\infty$. Thus, $\mathcal{U}^h(s,W)=-\infty$ when there exist no $C \subset \mathbb{S}^{n-1}$ with $|C|=N=|W|$ and $s(C)=s$.} weighted $h$-energy} among all weighted codes $(C,W)$ with fixed maximal inner product $s \in [-1,1)$ and fixed weights $W$.

Although the general linear programming theorems that will be presented in Sections 2 and 4 (Theorem \ref{t_lower} and 
Theorem \ref{t_upper}, respectively), 
hold for general potentials $h$, we will be especially interested when they are {\em absolutely monotone}
({\em strictly absolutely monotone}); that is, $h^{(i)}(t)\geq 0$, $i=0,1,\dots$ ($h^{(i)}(t)> 0$, $i=0,1,\dots$) for
all $t \in [-1,1)$ (all $t \in (-1,1)$). Commonly occuring absolutely monotone potentials include
\[ h(t)=[2(1-t)]^{1-n/2}, \mbox{ Newton potential},\]
\[ h(t)=[2(1-t)]^{-\alpha/2}, \ \alpha>0, \mbox{ Riesz potential}, \]
\[ h(t)=e^{-\alpha (1-t)}, \mbox{ Gaussian potential}, \]
\[ h(t)=-\log [2(1-t)], \mbox{ Logarithmic potential}. \]
As observed by Cohn and Kumar \cite[p. 101]{CK}, if an absolutely monotone function fails to be strictly absolutely monotone, 
then it is a polynomial. 

As is often the case in the study of spherical codes, the Gegenbauer polynomials $P_i^{(n)}$ of respective degrees $i=0, 1, \dots$, 
play a useful role. They are orthogonal with respect to the measure 
\[d\mu_n(t):=\gamma_n (1-t^2)^{\frac{n-3}{2}}\, dt, \quad t\in [-1,1], \]
where
$ \gamma_n := \Gamma(\frac{n}{2})/\sqrt{\pi}\Gamma(\frac{n-1}{2})$
is a normalizing constant that makes $\mu_n$ a probability measure. We normalize these polynomials by $P_i^{(n)}(1)=1$. 
Note that $P_i^{(n)}(t)$ is exactly the Jacobi polynomial $P_i^{(\alpha,\beta)}(t)$ with parameters
$\alpha=\beta=(n-3)/2$ and the corresponding normalization.

Given a weighted code $(C,W)$ we consider its {\em weighted moments}
\begin{equation}\label{WeightedMoments} 
M_\ell(C,W):=\sum_{i,j=1}^{|W|} w_iw_jP_\ell^{(n)}(x_i \cdot x_j), \ \ \ell \geq 1.
\end{equation}
It follows from the positive definiteness of the Gegenbauer polynomials (see e.g. \cite{Sch42}, \cite[Chapter 5]{BHS}) that $M_\ell(C,W) \geq 0$ for every positive integer $\ell$. The case
of equality for some $\ell$ is especially interesting. 

\begin{definition} \label{w-des} 
A weighted spherical code $(C,W)$ is called a {\em weighted spherical design of strength $\tau$} (or {\em a weighted spherical $\tau$-design})
if its first $\tau$ weighted moments are zero; i.e.,
\[ M_\ell(C,W)=0 \ \ \mbox{ for } 1 \leq \ell \leq \tau. \]
\end{definition}

In the equi-weighted case $w_1=\cdots=w_N=1/N$ one obtains the classical spherical designs introduced in the seminal paper of Delsarte, Goethals, and Seidel \cite{DGS} from 1977. The weighted case can be traced back to the 1960's and 70's when cubature formulas (or numerical integration schemes with
some degree of precision) for approximate calculation of multiple integrals on $\mathbb{S}^{n-1}$ were investigated \cite{S62,S71,S74,Sal78,GS81}. 
The related weighted designs in $\mathbb{R}^n$ and Euclidean designs on several concentric spheres were first considered in \cite{NS88} (see also the
comprehensive survey \cite{Ban09}).

Utilizing linear programming, we shall obtain bounds for the quantities 
$\mathcal{E}^h(W)$ given by \eqref{min-energy} and $\mathcal{U}^h(s,W)$ given by \eqref{min-w}. Our bounds 
are valid for all {\em absolutely monotone potentials} $h$ and are universal in the sense of Levenshtein (there is a bound for {\em every} 
weighted code) and in the sense of defining parameters (nodes and weights) 
that are {\em independent} of the potential function. Also, assuming the existence of a code that attains our bound for some absolutely monotone 
$h$, that code will attain our bound for {\em all absolutely monotone $h$}. That is, our bounds are universal in the sense of Cohn-Kumar. 
We shall also obtain bounds in the special cases when the codes are weighted spherical designs. 

The universal lower bounds (ULB, Theorem \ref{ulb}) on $\mathcal{E}^h(W)$ and universal upper
bounds (UUB, Theorem \ref{uub}) on $\mathcal{U}^h(s,W)$ are derived as certain solutions of linear 
programs that arise naturally as generalizations of the equi-weighted frameworks 
from \cite{BDHSS-ca} for the ULB and from \cite{BDHSS-dcc} for the UUB (see also \cite{BDHSS-amp} for such bounds in 
polynomial metric spaces). We
present examples of ULB and UUB for some special weighted codes that have previously attracted attention  
 for their high degree of precision when used in cubature formulas; see Sobolev \cite{S62},
Goethals and Seidel \cite{GS81}, Waldron \cite{W18}, and Godsil \cite[Theorem 3.2]{G88-89}. 

The paper is organized as follows. In Section 2 we develop the general linear programming technique for obtaining lower bounds on 
the energy of weighted codes and formulate a linear program to obtain bounds that are in a sense optimal. Then we solve that program in a
certain class of polynomials to derive our universal lower bound for the quantity $\mathcal{E}^h(W)$. In the remainder of the section we consider weighted spherical designs and present examples of our lower 
bounds in two cases where good weighted spherical designs can be easily constructed -- a weighted pentakis dodecahedron
(weighted union of an icosahedron with a dodecahedron), which is a weighted spherical 9-design, and a weighted union of a 
cube with a crosspolytope in $n$ dimensions, which is a weighted spherical 5-design. 
In Section 3 we prove a necessary condition for optimality of our lower bounds, which is also sufficient for
strictly absolutely monotone potentials with positive derivatives at $-1$. Section 4 is devoted to the counterpart technique for obtaining upper bounds for the 
quantity $\mathcal{U}^h(s,W)$. We develop the corresponding linear programming framework and propose a construction 
of a solution which might be optimal. The same two examples are considered from the upper bounds case point of view. 
In Section 5 we derive ULB and UUB for weighted spherical designs. The design 
properites allow us to establish both bounds for broader classes of potentials and to improve the UUB. 

\section{Universal lower bounds for energy of weighted codes}

\subsection{A general linear programming lower bound for weighted codes}

Given a potential function $h$, we consider the set of polynomials 
\[ L_h^{(n)}:= \left\{f(t)=\sum_{i=0}^{{\rm deg}(f)} f_i P_i^{(n)}(t) : f(t) \leq h(t), t \in [-1,1), f_i \geq 0, i\geq 1\right\}, \]
where $\left\{P_i^{(n)}\right\}_{i=0}^\infty$ are the Gegenbauer polynomials as defined in the Introduction.
The set $L_h^{(n)}$ will be the feasible domain for our
linear programming lower bounds for $\mathcal{E}^h(W)$. Without loss of generality we may assume that $f(1)>0$. 

The following theorem is a useful folklore result of Delsarte-Yudin type \cite{Del1,Del2,DGS,Y,KY97}. 

Hereafter we denote the common size of $C$ and $W$ by $N$.

\begin{theorem} \label{t_lower}
If $f(t)=\sum_{\ell =0}^{\deg(f)} f_\ell P_\ell^{(n)}(t) \in L_h^{(n)}$, then for every $(C,W)$ code on $\mathbb{S}^{n-1}$ 
\begin{equation} \label{DY1}
E_h(C,W) \geq E_f(C,W) \geq f_0-f(1)\sum_{i=1}^{N} w_i^2.
\end{equation}
Consequently,
\begin{equation} \label{DY_Bound}\mathcal{E}^h(W) \geq \sup_{f \in L_h^{(n)}} 
\left(f_0-f(1)\sum_{i=1}^{N} w_i^2\right)=:{\rm ULB}(W,h). \end{equation}
If equality is attained throughout \eqref{DY1} for some $f$ and $(C,W)$ with $C=(x_1,x_2,\ldots,x_N)$, 
then $f(x_i \cdot x_j)=h(x_i \cdot x_j)$ for every $i \neq j$ and $f_\ell M_\ell(C,W)=0$ 
for every $\ell \in \{1,2,\ldots,\deg(f)\}$. 
\end{theorem}

\begin{proof} The first inequality in \eqref{DY1} follows obviously from $f \leq h$ in $[-1,1)$; i.e.,
\[ E_h(C,W) = \sum_{i \neq j}w_iw_j h(x_i \cdot x_j) \geq \sum_{i \neq j}w_iw_j f(x_i \cdot x_j)=E_f(C,W). \]
For the second inequality we estimate $E_f(C,W)$ from below as follows:
\begin{eqnarray*}
E_f(C,W) &=& \sum_{i,j}w_iw_j f(x_i \cdot x_j)-f(1)\sum_{i=1}^{N} w_i^2 \\
              &=& \sum_{\ell=0}^{\deg(f)} f_\ell \sum_{i,j} w_iw_j P_\ell^{(n)}(x_i \cdot x_j) 
                      -f(1)\sum_{i=1}^{N} w_i^2 \\ 
              &=& f_0+\sum_{\ell=1}^{\deg(f)} f_\ell M_\ell(C,W)-f(1)\sum_{i=1}^{N} w_i^2  \\
              &\geq & f_0-f(1)\sum_{i=1}^{N} w_i^2.
\end{eqnarray*}
Above we used that the coefficient of $f_0$ is $\sum_{i,j=1}^{N} w_iw_j=
\left(\sum_{i=1}^{N} w_i \right)^2=1$, the inequalities $f_\ell \geq 0$ for $\ell \geq 1$, 
and $\sum_{i,j} w_iw_j P_\ell^{(n)}(x_i \cdot x_j)=M_\ell(C,W) \geq 0$ because of the positive definiteness of the Gegenbauer polynomials. 
The conditions for equality follow immediately from the above. 
\end{proof}

\begin{remark}
We note that equality will hold for the last two quantities of the multi-display formula above if and only if  $f_\ell M_\ell (C,W)=0$ for all $\ell=1,\dots,{\rm deg} (f)$, which is true when the weighted code $(C,W)$ is a spherical weighted design of high enough strength. We shall use this observation in Section \ref{ULB_Design} to broaden the class of potentials for which our universal bounds hold when applied to weighted spherical designs.
\end{remark}

Of particular importance is the case when the supremum in \eqref{DY_Bound} is taken over the class of polynomials $L_h^{(n)}\cap \mathcal{P}_m$, where $ \mathcal{P}_m$ denotes the class of polynomials of degree at most $m$. This yields the linear program
\begin{equation}\label{ULB_LP} 
\left\{ 
 \begin{array}{ll}
{\rm maximize} & \displaystyle f_0-f(1)\sum_{i=1}^{N} w_i^2 \\[8pt]
{\rm subject\  to} & f\in L_h^{(n)}\cap \mathcal{P}_m.
 \end{array}  \right.
\end{equation}
For particular parameters $h$, $W$, and $m$ we shall obtain explicit solutions of this linear program. We shall denote the 
maximized objective function by ${\rm ULB}_m (W,h)$.
Note that 
\begin{equation} \label{ineq-w2}
S_W:=\sum_{i=1}^{N} w_i^2 \geq \frac{1}{{N}} 
\end{equation}
with equality if and only if $w_1=w_2=\cdots=w_N=1/N$ (i.e., in the classical case of equi-weighted code). 
Then it follows that
\[ f_0-f(1)\sum_{i=1}^{N} w_i^2 \leq f_0-\frac{f(1)}{N}, \]
where the right-hand side coincides exactly with the quantity that appears in the linear programming bound for the equi-weighted codes. 
This means that the bounds from Theorem \ref{t_lower} will be always less than the bounds for the corresponding equi-weighted 
case. In any case, it is important to note that the quantity 
\[ N_W:=\frac{1}{S_W}=1/\sum_{i=1}^{N} w_i^2 \]
 has to play an important role since it is going to determine the parameters (nodes and weights)
of the universal lower bound in the same way as the cardinality $N$ does in \cite{BDHSS-ca}.
Clearly, as the weights $w_i$ get closer in value to one another, as measured by the {\em variance} 
  \begin{equation}\label{Svar}
{\rm var}\,{W}:= ( {1}/{N})S_W-{1}/{N^2},
\end{equation} 
the quantity $N_W$ approaches $|W|=N$ from below. 
The inequality \eqref{ineq-w2}, written as $N \geq N_W$, means that $N_W$ is always less than or equal to the tuple size 
$N$ with equality only for equal weights and serves to replace $N$ in the framework from \cite{BDHSS-ca}. 
This discussion further justifies the bounds in this paper as natural generalizations of the ULB from \cite{BDHSS-ca} and UUB from \cite{BDHSS-dcc}. 
Note that in the UUB setting of Section 4 one additional constant, denoted by $N_q$, will appear, where $N_q>N \geq N_W$. Note that while $N$ is a natural number, $N_W$ and $N_q$ are in general real.

\subsection{Delsarte-Goethals-Seidel bound for minimal spherical designs and Levenshtein bound for maximal spherical codes}

The parameters of our energy bounds are closely related to two classical universal linear programming bounds, the Delsarte-Goethals-Seidel lower bound for minimal cardinality $\mathcal{B}(n,m)$ of spherical designs of given dimension $n$ and strength $m$ and the Levenshtein upper bound for the maximal cardinality $\mathcal{A}(n,s)$ of spherical codes of given dimension $n$ and maximal inner product $s$. 

The Delsarte-Goethals-Seidel bound \cite{DGS} states that 
\begin{equation} \label{DGS-bound}
\mathcal{B}(n,m) \geq D(n,m) := {n+k-2+\varepsilon \choose n-1}+{n+k-2 \choose n-1}
\end{equation}
for $m=2k-1+\varepsilon$, where $\varepsilon \in \{0,1\}$ just indicates the parity of $m$.
The numbers $D(n,m)$, $m=1,2,3,\ldots$, from \eqref{DGS-bound} define a partition of the positive integers in $[2,+\infty)$ into  countably many consecutive intervals. 
Therefore, there exists a unique positive integer $m$ such that $N_W \in (D(n,m),D(n,m+1)]$ (we will always assume that $N_W>2$). In what follows we will always assume that $N_W$ and $m$ are related this way.   

Following Levenshtein's notations from \cite{Lev92}, for $a,b \in \{0,1\}$ we denote by $t_k^{a,b}$ the largest zero of the 
Jacobi polynomial $P_k^{(a+(n-3)/2,b+(n-3)/2)}(t)$, $k \geq 1$, with $t_0^{1,1}:=-1$ by definition. For $m\in \mathbb{N}$, consider the interval 
\begin{eqnarray*}
  I_m :=
\left\{
\begin{array}{ll}
    \left [ t_{k-1}^{1,1},t_k^{1,0} \right ], & \mbox{if } m=2k-1, \\[6pt]
    \left [ t_k^{1,0},t_k^{1,1} \right ],      & \mbox{if } m=2k, \\
  \end{array}\right.
\end{eqnarray*}
(or $I_m:=[t_{k-1+\varepsilon}^{1,1-\varepsilon},t_{k}^{1,\epsilon}]$ with the short $\varepsilon$-notation).
The interlacing properties 
\[ t_{k-1}^{1,1}<t_k^{1,0}<t_k^{1,1}, \ \ k \geq 1, \] 
(see \cite[Sections 3, 4]{Lev92} or \cite[Lemmas 5.29, 5.30]{Lev98}) 
imply that the collection of intervals $\{I_m\}_{m=1}^\infty$ is well defined and constitute a partition of the interval $[-1,1)$ into countably many subintervals with non-overlapping interiors. The Levenshtein bound (obtained in 1979 \cite{Lev79}; see also \cite{Lev92,Lev98}) states that 
\begin{equation} \label{Lev-bound}
\mathcal{A}(n,s) \leq L_m(n,s), \ \ s \in I_m,
\end{equation}
where 
\[ L_m(n,s):= {k+n-3+\varepsilon \choose n-2} \left[ \frac{2k+n-3+2\varepsilon}{n-1}
-\frac{(1+s)^{\varepsilon}\left(P_{k-1+\varepsilon}^{(n)}(s)-
P_{k+\varepsilon}^{(n)}(s)\right)}
{(1-s)\left(\varepsilon P_{k}^{(n)}(s)+P_{k+\varepsilon}^{(n)}(s)\right)} \right]. \]

There is a strong relation between the Delsarte-Goethals-Seidel bound and the Levenshtein bound. For every fixed $m \in \mathbb{N}$,
the function $L_m(n,s)$ is continuous, strictly increasing, and maps the interval $I_m$ bijectively onto the interval $[D(n,m),D(n,m+1)]$. This 
relates the two partitions from above and, moreover, implies the four equalities  (recall that $m=2k-1+\varepsilon$)
\begin{equation}\label{L-DGS1}
L_{m}(n,t_{k-1+\varepsilon}^{1,1-\varepsilon}) = D(n,m), \ \ L_{m}(n,t_{k}^{1,\varepsilon}) = D(n,m+1),\ \ \varepsilon \in \{0,1\} 
\end{equation}
at the ends of the intervals $I_{m}$. The Levenshtein function $\mathcal{L}(n,s):[-1,1) \to [2,+\infty)$ is defined to coincide with $L_m(n,s)$ on $I_m$, $m \geq 1$.
It follows from \eqref{L-DGS1} that it is continuous, strictly increasing, and smooth in the interiors of the intervals $I_m$.

The Levenshtein bound $L_m(n,s)$ is valid (and optimal in a sense; see \cite{Sid80}, \cite[Section 4]{Lev92}, \cite[Theorem 5.39]{Lev98}) 
in the whole interval $s \in I_m$. It is defined via corresponding Levenshtein polynomial $f_{m}^{(n,s)}(t)$ and one has 
$L_m(n,s)=f_m^{(n,s)}(1)/f_0$ (see, e.g., \cite[Section 4]{Lev92} for detailed description of this correspondence). Note that 
\[ f_m^{(n,s)}(t)= (t-\alpha_{k-1+\varepsilon}) (t-\alpha_0)^{2-\varepsilon} \prod_{i=1}^{k-2+\varepsilon}(t-\alpha_i)^2, \] 
where the roots of $f_m^{(n,s)}$ will play important roles in our considerations (first of all, $\alpha_{k-1+\varepsilon}=s$).

\subsection{Generation of parameters}
We are now in a position to generate the necessary parameters for our universal lower bounds. 

Given a weighted spherical code $(C,W)$ we find the parameter $N_W>2$ and determine the unique positive integer $m=2k-1+\varepsilon$ such that 
\begin{equation} \label{int-sum}
D(n,m) < N_W \leq D(n,m+1).
\end{equation}
Next, we find the unique $s \in I_m=[t_{k-1+\varepsilon}^{1,1-\varepsilon},t_{k}^{1,\epsilon}]$ such that 
\begin{equation} \label{L}
L_m(n,s)=N_W
\end{equation}
and construct the corresponding Levenshtein polynomial $f_m^{(n,s)}(t)$. 
Then the roots of $f_m^{(n,s)}$, 
\[ \alpha_0<\alpha_1 < \cdots < \alpha_{k-1+\varepsilon}=s \]
(note that $-1 \leq \alpha_0$ with equality if and only if $\varepsilon=1$; i.e. if $m=2k$) will serve as nodes for both, quadrature \eqref{QR} and our Hermite interpolation. 
Alternatively, one can obtain the nodes $(\alpha_i)_{i=0}^{k-1+\varepsilon}$ as the roots of \eqref{L} considered as an equation in the variable $s$.

We also need weights $(\rho_i)_{i=0}^{k-1+\varepsilon}$ ($\rho_i$ corresponds to $\alpha_i$) which are computed by substituting the 
Lagrange basis polynomials $\ell_i(t)=\prod_{j \neq i} (t-\alpha_j)$ for 
$i=0,1,\ldots,k-1+\varepsilon$ in the quadrature formula \eqref{QR} below. 
Explicit formulas for $(\rho_i)_{i=0}^{k-1+\varepsilon}$ in the case $\varepsilon=0$ (this is for odd $m=2k-1$) 
can be found in \cite[Appendix A4]{BDL99} (see also the expressions in the proof of Theorem 5.39 from \cite{Lev98}). 
We note also that
\[ \rho_i = \int_{-1}^{1} \ell_i^2(t) d \mu_n(t)>0, \ \ i=0,1,\ldots,k-1+\varepsilon, \]
(see \cite[Theorem 5.39]{Lev98}) and the identity $\sum_{i=0}^{k-1+\varepsilon} \rho_i =1-1/N_W$ both hold by setting $f(t)=\ell_i^2(t)$ or $f(t) \equiv 1$, respectively, in the quadrature formula \eqref{QR} below.

Summarizing, the derivation of the necessary parameters proceeds as follows. Given a weighted code $(C,W)$, we find $N_W$ and the unique $m=2k-1+\varepsilon$ such that $N_W \in (D(n,m),D(n,m+1)]$. Then we solve \eqref{L} as an equation in $s$ and derive the parameters  
$(\alpha_i)_{i=0}^{k-1+\varepsilon}$ and $(\rho_i)_{i=0}^{k-1+\varepsilon}$.

\subsection{Universal lower bound for weighted codes}

It is crucial for our approach that the Gauss-type quadrature formula (cf. \cite[Theorem 5.39]{Lev98}) 
\begin{equation} \label{QR}
f_0=\frac{f(1)}{L_m(n,s)} + \sum_{i=0}^{k-1+\varepsilon} \rho_i f(\alpha_i) = 
            \frac{f(1)}{N_W} + \sum_{i=0}^{k-1+\varepsilon} \rho_i f(\alpha_i) 
\end{equation}
holds true for every polynomial $f(t)=f_0+\sum_{i=1}^{\deg(f)} f_iP_i^{(n)}(t)$ 
of degree at most $2k-1+\varepsilon$ (recall that
$N_W$ is less than but possibly close to $N$). The formula \eqref{QR} is called a {\em $1/N_W$-quadrature rule} in the framework from \cite{BDHSS-ca}. Note that we do not require the existence of  a  spherical code  with maximum inner product $s$ for such a quadrature to exist.

As in the equi-weighted case we will need two facts from the theory of orthogonal polynomials. 
Namely, it is well known that the Gegenbauer expansions of the polynomials $P_i^{(n)}(t)P_j^{(n)}(t)$ and 
$(t+1)P_i^{(n+2)}(t)P_j^{(n+2)}(t)$ have nonnegative coefficients for every $i,j$; see \cite{Gas70}.  
In \cite{Lev98}, these properties are called {\em Krein conditions} and {\em strengthened Krein conditions}, respectively (cf. \cite[Sections 3 and 5]{Lev98}).

We are now in a position to solve the linear program \eqref{ULB_LP}.

\begin{theorem} \label{ulb} (ULB for weighted codes) 
Let $W$ be a weight vector such that $N_W$ satisfies \eqref{int-sum} and let $h$ be absolutely monotone. Then
\begin{equation} \label{ulb-formula}
\mathcal{E}^h(W) \geq \sum_{i=0}^{k-1+\varepsilon} \rho_i h(\alpha_i), 
\end{equation}
where the parameters $(\alpha_i,\rho_i)_{i=0}^{k-1+\varepsilon}$ are defined as above. This bound cannot be 
improved by any polynomial from $L_h^{(n)} \cap \mathcal{P}_m$; i.e., the maximized objective function ${\rm ULB}_m (W,h)$ is equal to the right-hand side of \eqref{ulb-formula}.
\end{theorem}

\begin{proof}
Let $f$ be the unique Hermite interpolant to $h$ at the nodes $(\alpha_i)_{i=0}^{k-1+\varepsilon}$, each counted twice except 
for the case $\alpha_0=-1$ (equivalent to $m=2k$) which is 
counted once. Then $\deg(f) \leq m$, so \eqref{QR}  along with the interpolation conditions $f(\alpha_i)=h(\alpha_i)$, $i=0,1,\ldots,k-1+\varepsilon$, 
yields the equalities
\[ f_0-f(1)\sum_{i=1}^N w_i^2 = \sum_{i=0}^{k-1+\varepsilon} \rho_i f(\alpha_i)
= \sum_{i=0}^{k-1+\varepsilon} \rho_i h(\alpha_i)={\rm ULB}_m (W,h). \]
The error formula for the Hermite interpolation implies that for all $t\in [-1,1)$
\[ h(t)-f(t)=\frac{h^{(2k+\varepsilon)}(\xi)}{(2k+\varepsilon)!}(t-\alpha_0)^{2-\varepsilon} 
\prod_{i=1}^{k-1+\varepsilon}(t-\alpha_i)^2 , \quad \xi\in(-1,1). \]
As $h^{(2k+\varepsilon)}\geq 0$ on $[-1,1]$, we conclude that $f(t) \leq h(t)$ for every $t \in [-1,1)$. 

It remains to prove that $f$ is positive definite to complete the verification that $f \in L_h^{(n)}$.  

We first consider the case $\varepsilon=0$ (i.e., $m=2k-1$). Order the multiset of nodes as 
\[ (\alpha_0,\alpha_0,\alpha_1,\alpha_1,\ldots,\alpha_{k-1},\alpha_{k-1})=
 (t_1,t_2,\ldots,t_{2k-1},t_{2k}) \]
(i.e., $t_{2i+1}=t_{2i+2}=\alpha_i$ for $i=0,1,\ldots,k-1$). Then
the Newton interpolation formula (see, for example, \cite{B05}) 
\[ f(t)=h(t_1)+\sum_{r=1}^{2k-1} h[t_1,\ldots,t_{r+1}] \prod_{j=1}^{r} (t-t_j)  \]
implies that the polynomial $f$ is a nonnegative linear combination of the 
constant 1 and the partial products
\begin{equation} \label{p-prod-odd}
\prod_{j=1}^r (t - t_j), \ r=1,2,\ldots,2k-1. 
\end{equation}
Since 
\[ B_1(t-\alpha_0)(t-\alpha_1)\cdots(t-\alpha_{k-1})=P_k^{((n-1)/2,(n-3)/2)}(t)-\beta_1 P_{k-1}^{((n-1)/2,(n-3)/2)(t)},\]
for some positive constants $B_1$ and $\beta_1$, it follows from \cite[Theorem 3.1]{CK} that all polynomials 
\[ (t-\alpha_0)(t-\alpha_1)\cdots(t-\alpha_{i}), \ i=0,1,\ldots,k-2, \]
expand in the system $\{P_i^{((n-1)/2,(n-3)/2)}(t)\}$ with nonnegative coefficients. 
Since every polynomial $P_i^{((n-1)/2,(n-3)/2)}(t)$ is positive definite (cf. \cite[Eq. (5.119)]{Lev98}; this follows directly from
the Christoffel-Darboux formula which relates explicitly this polynomial to the Gegenbauer polynomials $P_j^{(n)}$, $j=0,1,\ldots,i-1$), 
the Krein condition implies that all partial products
\eqref{p-prod-odd} with $r \leq 2k-2$ are positive definite. The only remaining 
partial product (with $r=2k-1$ in \eqref{p-prod-odd})
is exactly the Levenshtein polynomial $f_{2k-1}^{(n,s)}(t)$ which is positive definite as
well (see, for example, \cite[Theorem 5.42]{Lev98}).
Therefore $f$ is positive definite and \eqref{ulb-formula} follows in this case. 


For $\varepsilon=1$ (i.e., $m=2k$) we need the strengthened Krein condition.
Now the interpolation nodes are ordered as 
\[ (\alpha_0=-1,\alpha_1,\alpha_1,\ldots,\alpha_{k},\alpha_{k})=
 (-1,t_1,t_2,\ldots,t_{2k-1},t_{2k}) \]
and our polynomial $f$ is a nonnegative linear combination of 1 
and the partial products
\begin{equation}
\label{p-prod-even}
(t+1)\prod_{j=1}^r (t - t_j), \ r=1,2,\ldots,2k-1.
\end{equation}
Theorem 3.1 from \cite{CK} now implies that all polynomials 
\[ (t-\alpha_1)(t-\alpha_2)\cdots(t-\alpha_{i}), \ i=1,2,\ldots,k-1, \]
expand in the system $\{P_i^{(n+2)}(t)\}$ with nonnegative coefficients because
\[ B_2(t-\alpha_1)(t-\alpha_2)\cdots(t-\alpha_{k})=P_k^{((n-1)/2,(n-1)/2)}(t)-\beta_2 P_{k-1}^{((n-1)/2,(n-1)/2)(t)},\]
for some positive constants $B_2$ and $\beta_2$.
Then all partial products from \eqref{p-prod-even} with $r \leq 2k-2$ expand with
nonnegative coefficients in $(t+1)P_i^{(n+2)}(t)P_j^{(n+2)}(t)$ and the
strengthened Krein condition completes the argument.
Finally, in the case $r=2k-1$ we obtain exactly the Levenshtein polynomial 
$f_{2k}^{(n,s)}(t)$ 
which is positive definite and this is exactly what is needed to complete the proof that $f \in L_h^{(n)} \cap 
\mathcal{P}_m$. Thus, \eqref{ulb-formula} follows in this case as well.


If $g(t)=\sum_{i=0}^{\deg(g)} g_i P_i^{(n)}(t)$ is a polynomial from 
$L_h^{(n)} \cap \mathcal{P}_m$ and $f$ is defined as above, then 
\eqref{QR} can be applied with $g$ to see that the lower bound provided by $g$ satisfies
\[ g_0-g(1)\sum_{i=1}^{N} w_i^2 = \sum_{i=0}^{k-1+\varepsilon} \rho_i g(\alpha_i)
\leq \sum_{i=0}^{k-1+\varepsilon} \rho_i h(\alpha_i)=f_0-f(1)\sum_{i=1}^{N} w_i^2={\rm ULB}_m (W,h), \]
which completes the proof. 
\end{proof}

We next establish the monotonicity of ${\rm ULB}_m(W,h)$ in $N_W$. 

\begin{theorem}Let $V=(v_1,\dots,v_{\widetilde{N}})$ and $W=(w_1,\dots,w_N)$ be two weight vectors with positive components 
such that $\sum_{i=1}^{|V|}  v_i=\sum_{i=1}^{|W|} w_i=1$, and suppose that $N_V<N_W$ (equivalently, $S_V>S_W$). 
Let $\widetilde{m}$ and $m$ be the positive integers associated with $V$ and $W$, respectively, via \eqref{int-sum}. Then $m \geq \widetilde{m}$ and 
$${\rm ULB}_m (W,h)>{\rm ULB}_{\widetilde{m}} (V,h).$$
If $m=\widetilde{m}$, then the nodes $(\alpha_i)_{i=0}^{k-1+\varepsilon}$ for $N_W$ are strictly greater than the corresponding nodes 
for $N_V$. 
\end{theorem}
\begin{proof} The inequality $N_V<N_W$ implies via \eqref{int-sum} that $m \geq \widetilde{m}$.
If the equality holds, \eqref{L} and the monotonicity of the Levenshtein function $\mathcal{L}(n,s)$ imply the monotonicity of the nodes $\alpha_i$; i.e., 
they are increasing with $s=\alpha_{k-1+\varepsilon}$ 
which is increasing with $N_W$ (see \cite{BDL99}).

Let $f$ and $g$ be the (unique) polynomial solutions of \eqref{ULB_LP} 
associated with $W$ and $V$, respectively. Then, as 
$g\in L_h^{(n)}\cap\mathcal{P}_{\widetilde{m}} \subseteq L_h^{(n)}\cap\mathcal{P}_m$,
the optimality of $f$ over $L_h^{(n)}\cap \mathcal{P}_m$ yields 
$$ {\rm ULB}_{m}(W,h)=f_0-\frac{f(1)}{N_W} \geq g_0-\frac{g(1)}{N_W}
> g_0-\frac{g(1)}{N_V}= {\rm ULB}_{\widetilde{m}}  (V,h).$$
Note that $g$ has positive Gegenbauer coefficients and $g(1)=g_0+\cdots+g_{\widetilde{m}}>0$.
\end{proof}

\subsection{Low degrees bounds} \label{small-d-lower}

We proceed with derivation of more explicit forms of the first degrees ULB. 
For degree one ULB, we have (see \cite[Table 6.1]{Lev98} for the first five bounds $L_m(n,s)$) $L_1(n,s)=(s-1)/s=N_W$ (note that $s \in [-1,-1/n]$
is negative). This gives 
\[ \alpha_0=-\frac{1}{N_W-1}, \ \ \rho_0 = -\frac{1}{\alpha_0N_W} = \frac{N_W-1}{N_W}, \] 
where $2 \leq N_W (\leq N) \leq n+1$. 
Therefore,
\[ \mathcal{E}^h(W) \geq {\rm ULB}_1 (W,h)= \rho_0 h(\alpha_0) = \frac{N_W-1}{N_W} \cdot h\left(-\frac{1}{N_W-1}\right), \ 2 \leq N_W \leq n+1 . \]

The degree two ULB is computed via $L_2(n,s)= 2n(1-s)/(1-ns)=N_W$, where $s \in [-1/n,0]$ 
(corresponding to the restrictions $n+1 \leq N_W \leq 2n$), giving the parameters 
\[ \alpha_0=-1, \ \alpha_1 = -\frac{2n-N_W}{n(N_W-2)}, \] 
\[ \rho_0 = \frac{N_W-n-1}{(n+1)N_W-4n}, \ \rho_1 = \frac{n(N_W-2)^2}{N_W((n+1)N_W-4n)}. \] 
This leads to the bound
\begin{eqnarray*}
\mathcal{E}^h(W) &\geq& {\rm ULB}_2 (W,h)=\rho_0 h(\alpha_0) + \rho_1 h(\alpha_1) \\
                       &=& \frac{N_W(N_W-n-1)h(-1)+n(N_W-2)^2h\left(-\frac{2n-N_W}{n(N_W-2)}\right)}{N_W((n+1)N_W-4n)} ,
\end{eqnarray*}
where $n+1 \leq N_W \leq 2n$.

The degree three ULB is already too complicated to be explicitly stated here. However, for particular suitable potentials, like the Fejes T{\'o}th potential
(see the next paragraph), where the potential and the formulas for the weights $(\rho_i)_{i=0}^{k-1}$ via the nodes $(\alpha_i)_{i=0}^{k-1}$
(this is for $\varepsilon=0$, i.e. $m=2k-1$) from \cite[Appendix]{BDL99} are relatively simple and fit well each other, the calculations
of ULB$_3(W,h)$ are still doable.  

The potential $h(t)=-\sqrt{2(1-t)}$ fits in the above scheme because $2+h(t)$ is absolutely monotone. This potential corresponds to the Fejes T{\'o}th
problem \cite{FejesToth1956} and it has been studied by many authors (see, for example, \cite{BilykMatzke,BBS23} and references therein).
The degrees 1-3 ULB for weighted codes and this particular potential (and their asymptotic consequences) can be extracted from \cite{BBS23} 
simply by replacing $N$ by $N_W$ and dividing by $N_W^2$.

\subsection{Examples}

In contrast to difficulties for derivation of more explicit analytic expressions of ULB$_m(W,h)$ for $m \geq 3$, the numerical calculations of 
bounds for given $n$, $h$, and $W$ can be easily programmed. 
In this subsection we present several examples, where the ULB and the actual weighted energy are computed. 

In all computations here and below we used two independent programs, one in Maple and one in Mathematica. Both programs produced the
same numbers with 50-digits precision. For the sake of short presentation, we truncate or round up (depending on whether the bounds are lower
or upper, respectively) the real numbers giving the weighted energy ULB and UUB. The values of the remaining real parameters ($N_W$,
nodes, weights) are rounded to the fourth digit.

\begin{example} \label{32} (Union of icosahedron and dodecahedron.)
Let $C_{32} \subset \mathbb{S}^2$ consist of the 12 vertices of an icosahedron, each of weight $w_I=20/(21 \cdot 32)=5/168$, and the 20 vertices 
of a dodecahedron, each of weight $w_D=36/(35 \cdot 32)=9/280$. The vertices of the icosahedron are the centers of the spherical caps defined by the 
twelve pentagonal faces of the dodecahedron. In geometry, this is called a pentakis dodecahedron or a kisdodecahedron.
Note that $(C_{32},W)$ is a weighted spherical 9-design (see \cite[Section 5]{GS81}, \cite[Example 3.6]{HW21}). 

We proceed with computations of the actual weighted energy of $(C_{32},W)$ and 
the corresponding ULB$_9(W,h)$ for the potential function $h(t)=1/\sqrt{2(1-t)}$.

The weighted energy of $(C_{32},W)$ is computed from the information about its structure from Table \ref{tab:table1}. There are two types of 
points, I and D, respectively, according to whether they belong to the icosahedron or the dodecahedron, which define the two different distance
distributions (the last two rows of Table  \ref{tab:table1}). To shorten the notation we set
\[ a:=\frac{\sqrt{1-2/\sqrt{5}}}{\sqrt{3}}, \ \ b:=\frac{\sqrt{1+2/\sqrt{5}}}{\sqrt{3}}. \]

\begin{table}[h!]
  \centering
  \caption{Structure of $(C_{32}, W)$.}
  \label{tab:table1}
  \begin{tabular}{|c|c|c|c|c|c|c|} \hline
     & \multicolumn{6}{|c|}{Inner products} \\ \cline{2-7}
    & $-1$ & $\pm 1/\sqrt{5}$ & $\pm a$ & $\pm b$ & $\pm 1/3$ & 
$\pm \sqrt{5}/3$\\
    \hline
 Type     & \multicolumn{6}{|c|}{Number of points} \\ \hline
    I & $1$ & $5$ & $5$ & $5$ & $0$ &  $0$\\ \hline
    D & $1$ & $0$ & $3$ & $3$ & $6$ & $3$\\ \hline
  \end{tabular}
\end{table}

Then,
\begin{eqnarray*}
E_h(C_{32},W) &=& \sum_{i \neq j} w_iw_j h(x_i \cdot x_j) \\
&=& 12w_I^2\left(h(-1)+5h(-1/\sqrt{5})+5h(1/\sqrt{5})\right) \\
      &&   +\, 120w_Iw_D\left(h(a)+h(-a)+h(b)+h(-b)\right)\\
&& + \, 20w_D^2\left(h(-1)+6h(-1/3)+6h(1/3)+3h(-\sqrt{5}/3)+3h(\sqrt{5}/3)\right) \\
&\approx& 0.8050318.
\end{eqnarray*}

We have $N_W=1/\sum_{i=1}^{32} w_i^2=735/23 \approx 31.9565217$ which is very close to the 
cardinality 32 of $C_{32}$. Thus, we expect good bounds. We compute the ULB for $n=3$, $N_W$, and the Coulomb potential $h(t)=1/\sqrt{2(1-t)}$.
Since $N_W$ belongs to the interval $(D(3,9),D(3,10)]=(30,36]$ (note that $N=32$ is in the same interval), we 
solve the equation $L_9(3,s)=N_W$ to derive the parameters 
$(\alpha_i,\rho_i)_{i=0}^4$ with approximate values as shown in Table \ref{tab:table2}.

\begin{table}[h!]
  \centering
  \caption{Parameters $(\alpha_i,\rho_i)_{i=0}^4$ for $(n,N,N_W)=(3,32,735/23)$.}
  \label{tab:table2}
  \begin{tabular}{|c|c|c|c|c|c|} \hline
    $i$ & $0$ & $1$ & $2$ & $3$ & $4$ \\
    \hline
    $\alpha_i$ & $-0.9412$ & $-0.6741$ & $-0.2109$ & $0.3281$ & $0.7793$ \\ \hline
    $\rho_i$ & $0.0771$ & $0.1889$ & $0.2636$ & $0.2612$ & $0.1777$ \\ \hline
  \end{tabular}
\end{table}

Therefore,
\[ \mathcal{E}^h(W) \geq {\rm ULB}_9(W,h)=\sum_{i=0}^{4} \rho_i h(\alpha_i) \approx 0.804786,\] 
which is very close (within $0.03\%$) to the actual $h$-energy $\approx 0.8050318$ of $(C_{32},W)$. 

It is worth mentioning that in the equi-weighted case, the Coulomb energy of $(C_{32},(1/32)^{32})$ is $E_h(C_{32},((1/32)^{32})) \approx 0.8052$ and the universal lower bound from \cite{BDHSS-ca} for $(n,N)=(3,32)$ and the same $h$ is $\approx 0.8049$. \hfill $\Box$
\end{example}

\begin{example} \label{cube+biortho} (Union of cube and cross-polytope.)
We consider a weighted code $(C_{qp},W)$ comprised of the union of a cube and a cross-polytope on 
$\mathbb{S}^{n-1}$
defined by their duality; i.e., each pair of antipodal vertices of the cross-polytope defines
a symmetry axis of two opposite facets of the cube. 
Each point of the cross-polytope has weight $w_p:=1/(2n+n^2)$ and each point of the 
cube has weight $w_c:=n^2/2^n(2n+n^2)$. We see that the sum of weights 
of the union is 1. Furthermore, $(C_{qp},W)$ is a weighted 5-design on $S^{n-1}$ for $n \geq 3$, as we now show. 

Indeed, since $C_{qp}$ is antipodal and the weights of points in each antipodal pair are equal, all odd (in particular the first, third, and fifth) 
weighted moments \eqref {WeightedMoments} of $(C_{qp},W)$ equal zero. To show that the fourth weighted moment is zero, we denote by $Q$ 
the set of vertices of the cube and by $P$ the set of vertices of the cross-polytope in $(C_{qp},W)$ and observe that for every 
$x=(x_1,\ldots,x_n)\in \mathbb{S}^{n-1}$,
\begin{equation}\label {qm}
\begin{split}
U_4(x):&=w_c\sum\limits_{y\in Q}(x\cdot y)^4+w_p\sum\limits_{y\in P}(x\cdot y)^4\\
&=\frac {w_c}{n^2}\sum\limits_{\sigma_1,\ldots,\sigma_n\in \{-1,1\}}(\sigma_1x_1+\cdots+\sigma_nx_n)^4+2w_p\sum\limits_{i=1}^{n}x_i^4 \\
&=\frac {2^nw_c}{n^2}\left(\sum\limits_{i=1}^{n}x_i^4+3\sum\limits_{i=1}^{n}\sum\limits_{j=1\atop j\neq i}^{n}x_i^2x_j^2\right)+2w_p\sum\limits_{i=1}^{n}x_i^4\\
&=\frac {3}{2n+n^2}\left(\sum\limits_{i=1}^{n}x_i^4+\sum\limits_{i=1}^{n}\sum\limits_{j=1\atop j\neq i}^{n}x_i^2x_j^2\right)\\
&=\frac {3}{2n+n^2}\left(x_1^2+\cdots+x_n^2\right)^2=\frac {3}{2n+n^2}=\gamma_n\int_{-1}^{1}t^4(1-t^2)^{\frac {n-3}{2}}\ \! dt.
\end{split}
\end{equation}

Using a similar argument, we can also show that
$$
U_2(x):=w_c\sum\limits_{y\in Q}(x\cdot y)^2+w_p\sum\limits_{y\in P}(x\cdot y)^2=\frac {1}{n}= 
\gamma_n\int_{-1}^{1}t^2(1-t^2)^{\frac {n-3}{2}}\ \! dt,\ \ \ x\in \mathbb{S}^{n-1}.
$$
Consequently,
\begin {equation}\label {qn}
w_c\sum\limits_{y\in Q}P_2^{(n)}(x\cdot y)+w_p\sum\limits_{y\in P}P_2^{(n)}(x\cdot y)=0,\ \ \ x\in \mathbb{S}^{n-1}.
\end {equation}

Finally, expressing $P_4^{(n)}$ from the Gegenbauer expansion of $t^4$ and using \eqref {qm}, \eqref {qn}, and the fact that \eqref {qm} equals the constant term in the Gegenbauer expansion of $t^4$, we obtain that
$$
w_c\sum\limits_{y\in Q}P_4^{(n)}(x\cdot y)+w_p\sum\limits_{y\in P}P_4^{(n)}(x\cdot y)=0,\ \ \ x\in \mathbb{S}^{n-1},
$$
which together with \eqref {qn} implies that the second and the fourth weighted moments \eqref{WeightedMoments} of $(C_{qp},W)$ equal zero.
Thus, $(C_{qp},W)$ is a weighted 5-design. 
 
For $n=2$, $(C_{qp},W)$ is a regular 8-gon and the weights are equal; i.e. it is a tight spherical 7-design. Therefore, it is a sharp spherical
code and a universally optimal configuration \cite{CK}.

In three and four dimensions, the codes $(C_{qp},W)$ can be described as follows. On $\mathbb{S}^2$, each point of the cross-polytope will have weight $1/15$ 
and each point of the cube will have weight $3/40$, giving a weighted spherical $5$-design of $14$ points; 
on $\mathbb{S}^3$, each point of the cross-polytope will have weight $1/24$ and each point of the cube will 
also have weight $1/24$ (thus, we obtain a $24$-cell, an equi-weighted spherical 5-design; see \cite{CCEK}). Note that there exist no equi-wegthed
spherical $5$-designs with 13 points \cite{BBD99} and the existence of such designs with 14 points is undecided. 

For any $h$, the actual weighted $h$-energy of $(C_{qp},W)$ is 
\begin{eqnarray*}
E_h(C_{qp},W) &=& 2nw_p^2\left(h(-1)+(2n-2)h(0)\right) \\
&& +\, 2^{n+1}nw_pw_c\left(h\left(\frac{1}{\sqrt{n}}\right)+
           h\left(-\frac{1}{\sqrt{n}}\right)\right) \\
&& +\, 2^nw_c^2\sum_{k=0}^{n-1} {n \choose k}h\left(-1+\frac{2k}{n}\right).
\end{eqnarray*}

The ULB for the corresponding parameters $(n,|C_{qp}|=2n+2^n,N_W)$, where
\[ N_W=\frac{1}{\sum_{i=1}^{2n+2^n} w_i^2}=\frac{1}{2nw_p^2+2^nw_c^2}=\frac{n(n+2)^22^n}{n^3+2^{n+1}} \]
(note that $N_W \in (D(n,5),D(n,6)]$ for $3 \leq n \leq 6$ only), is computed as follows. We solve
\[ L_5(n,s)=N_W \iff \frac{\left((n+2)(n+3)s^2+4(n+2)s-n+1\right)(1-s)}{2s\left(3-(n+2)s^2\right)}=\frac{(n+2)^22^n}{n^3+2^{n+1}} \]
to obtain the nodes $(\alpha_i)_{i=0}^2$. Then the quadrature weights $(\rho_i)_{i=0}^2$ are computed by setting $f$ equal to the Lagrange basis
polynomials in \eqref{QR} or by the known formulas
\[ \rho_0=-\frac{(1-\alpha_1^2)(1-\alpha_2^2)}{\alpha_0N_W(\alpha_0^2-\alpha_1^2)(\alpha_0^2-\alpha_2^2)}, \ \ 
\rho_1=-\frac{(1-\alpha_0^2)(1-\alpha_2^2)}{\alpha_1N_W(\alpha_1^2-\alpha_0^2)(\alpha_1^2-\alpha_2^2)} \]
from \cite{BDL99} and the relation $\rho_0+\rho_1+\rho_2=1-1/N_W$.

The ULB in dimensions $2 \leq n \leq 7$, calculated for the absolutely monotone potential 
\[ h(t)=\frac{1}{(2(1-t))^{(n-2)/2}}, \] 
are shown in 
the sixth column of Table \ref{tab:table3}. It is ULB$_7(W,h)$ for $n=2$, ULB$_5(W,h)$ for $3 \leq n \leq 6$ and  ULB$_6(W,h)$ for $n=7$. Note that 
the bound ULB$_7(W,h)$ is attaned for $n=2$, where it coincides with the ULB for the equi-weighted case \cite{BDHSS-ca} (recall that 
the attaining $(C_{qp},W)$ is an equi-weighted regular 8-gon). 

\begin{table}[h!]
  \centering
  \caption{Approximate parameters and ULB for $(n,N,N_W)=(n,2n+2^n,N_W)$, $2 \leq n \leq 7$, $h(t)=(2(1-t))^{-(n-2)/2}$.}
  \label{tab:table3}
  \begin{tabular}{|c|c|c|c|c|c|c|} \hline
    $n$ & $N_W$ & $N$ & $(\alpha_i)$ & $(\rho_i)$ & $ULB$ & Energy of $(C_{qp},W)$  \\
\hline
          & & & $-1$ & $1/8$ &  &   \\ 
    $2$ & $8$ & $8$ & $-\sqrt{2}/2$ & $1/4$ & $0.875$ & $0.875$  \\ 
          & & & $0$ & $1/4$ &  &   \\ 
          & & & $\sqrt{2}/2$ & $1/4$ &  &   \\ 
    \hline
          & & & $-0.8580$ & $0.1832$ &  &   \\ 
    $3$ & $13.95$ & $14$ & $-0.2701$ & $0.3832$ & $0.7058$ & $0.7070$  \\ 
          & & & $0.5225$ & $0.3618$ &  &   \\ 
\hline
          & & & $-0.8173$ & $0.1384$ &  &   \\ 
    $4$ & $24$ & $24$ & $-0.2575$ & $0.4339$ & $0.5781$ & $0.5798$  \\ 
          & & & $0.4749$ & $0.3858$ &  &   \\ 
\hline
          & & & $-0.7428$ & $0.1424$ &  &   \\ 
    $5$ & $41.48$ & $42$ & $-0.1910$ & $0.4680$ & $0.4825$ & $0.4901$  \\ 
          & & & $0.4684$ & $0.3653$ &  &   \\ 
\hline
          & & & $-0.6753$ & $0.1540$ &  &   \\ 
    $6$ & $71.44$ & $76$ & $-0.1327$ & $0.4996$ & $0.4074$ & $0.4314$  \\ 
          & & & $0.4705$ & $0.3323$ &  &   \\ 
\hline
          & & & $-1$ & $0.0022$ &  &   \\ 
    $7$ & $121.16$ & $142$ & $-0.5936$ & $0.1785$ & $0.3462$ & $0.3993$  \\ 
          & & & $-0.0772$ & $0.5165$ &  &   \\ 
          & & & $0.4748$ & $0.2944$ &  &   \\ 
\hline
  \end{tabular}
\end{table}

We remark that the $h$-energy in the equi-weighted case and the corresponding ULB from \cite{BDHSS-ca} for the 14-point $C_{qp}$ 
in three dimensions are $\approx 0.70757$ and $\approx 0.70629$, respectively. \hfill $\Box$
\end{example}

Many other examples from the classical sources \cite{S62,S71,S74,Sal78,GS81} can be similarly explored. Also, one can 
derive upper bounds from the next section and bounds for weighted spherical designs in the last section.

\section{On the optimality of the ULB} 

We showed in Theorem  \ref{ulb} that the bound \eqref{ulb-formula} cannot be improved by using polynomials of degree $m$ or less. 
We shall extend this result by proving a necessary and sufficient condition for the optimality of \eqref{ulb-formula} among 
polynomials from the whole set $L_h^{(n)}$. 

With  parameters as in Theorem \ref{ulb} we define the functions
\begin{equation} \label{test-functions}
Q_j(n,s):= \frac{1}{N_W}+\sum_{i=0}^{k-1+\varepsilon}\rho_i P_j^{(n)}(\alpha_i), \quad j \geq 1,
\end{equation}
and call them {\em test functions}. The name will be justified by Theorem \ref{THM-TF}.

It is easy to see that $Q_j(n,s) \equiv 0$ for $j \in \{1,2,\ldots,m\}$ (because in this case the right-hand side in \eqref{test-functions} is 
equal, via the quadrature \eqref{QR}, to the coefficient $f_0$ of the Gegenbauer polynomial $P_j^{(n)}$, which clearly equals $0$). 
The next theorem, which is a weighted analog of Theorems 2.6 and 4.1 from \cite{BDHSS-ca} (see also Theorem 3.1 in \cite{BDB96} and
Theorem 5.47 in \cite{Lev98}), shows that the values of the test functions 
$Q_j(n,s)$ for $j \geq m+1$ are as meaningful as their signs are. 

\begin{theorem}\label{THM-TF} Let $h$ be absolutely monotone. For given $n$, $N_W$, and parameters $m$ and 
$(\alpha_i,\rho_i)_{i=0}^{k-1+\varepsilon}$ as in Theorem  \ref{ulb}, the following is true. 

{\rm (a)} If $Q_j(n,s) \geq 0$ for every positive integer $j$, then the bound ULB$_{m}(W,h)$ 
cannot be improved by any polynomial from $L_h^{(n)}$; i.e.,
\[ \mathrm{ULB} (W,h)=\mathrm{ULB}_m (W,h). \]

{\rm (b)} If $h$ is strictly absolutely monotone, $h^{(i)}(-1)>0$ for all $i \geq 0$, and $Q_j(n,s) < 0$ for some positive 
integer $j \geq m+1$, then there exists a polynomial from $L_h^{(n)}$ of degree $j$ that gives 
a bound on $\mathcal{E}^h(W)$ better than ULB$_{m}(W,h)$; i.e.,
\[ \mathrm{ULB} (W,h) > \mathrm{ULB}_m (W,h). \]
\end{theorem}

\begin{proof}
(a) Let us assume that $f \in L_{h}^{(n)}$ has degree $d \geq m+1$ (the case $d \leq m$ is covered by Theorem  \ref{ulb}). 
We decompose $f$ as
\begin{equation}
\label{n1}
f(t)= g(t)+\sum_{j=m+1}^d  f_j P_j^{(n)}(t), 
\end{equation}
where $g \in \mathcal{P}_m$. Note that $f_0=g_0$ and $f_j\ge 0$ for every relevant $j$. Using the 
quadrature \eqref{QR} for $h$ and the representation \eqref{n1}, we consecutively reorganize and finally estimate from above the bound
generated by $f$ as follows:
\begin{eqnarray*}
f_0- \frac{f(1)}{N_W} &=& g_0 - \frac{f(1)}{N_W} \\
    &=& \frac{g(1)}{N_W}+\sum_{i=0}^{k-1+\varepsilon} \rho_i g(\alpha_i) -\frac{1}{N_W} \cdot \left(g(1)+\sum_{j=m+1}^d  f_j \right) \\
    &=& \sum_{i=0}^{k-1+\varepsilon} \rho_i \left(f(\alpha_i)-\sum_{j=m+1}^d f_j P_j^{(n)}(\alpha_i)\right)-\frac{1}{N_W} \cdot\sum_{j=m+1}^d  f_j \\
    &=&\sum_{i=0}^{k-1+\varepsilon} \rho_i  f(\alpha_i)-\sum_{j=m+1}^d  f_j \left(\frac{1}{N_W}+ \sum_{i=0}^{k-1+\varepsilon}\rho_i P_j^{(n)}(\alpha_i) \right)\\
    &=& \sum_{i=0}^{k-1+\varepsilon} \rho_i  f(\alpha_i)-\sum_{j=m+1}^d f_jQ_j(n,s) \\
    &\le& \sum_{i=0}^{k-1+\varepsilon} \rho_i  h(\alpha_i)=\mathrm{ULB}_m (W,h),
\end{eqnarray*}
 where, for the last inequality, we used $f_i \geq 0$ and $Q_j(n,s) \geq 0$ for $j=m+1,\ldots,d$.

(b) We shall find an improvement of the bound ULB$_{m}(W,h)$ by using the polynomial
\[ f(t)=g(t)+\mu P_j^{(n)}(t), \]
where the constant $\mu >0$ and the polynomial $g(t)\in \mathcal{P}_{m}$ will be suitably chosen. 

We first define an auxiliary potential 
\[ \widetilde{h}(t):= h(t)-\mu P_j^{(n)}(t), \]
where $\mu>0$ is chosen in such a way that the derivatives $\widetilde{h}^{(i)}(t) \geq 0$ on $[-1,1)$ for all $i=0,1,\dots,j$. Since 
$\widetilde{h}^{(i)}(t)=h^{(i)}(t)>0$ for $i \geq j+1$, this choice of $\mu$ makes the new potential $\widetilde{h}(t)$ absolutely monotone. 

Next, we choose the polynomial $g(t)$ as the Hermite interpolant of $\widetilde{h}$ at the nodes $(\alpha_i)_{i=0}^{k-1+\varepsilon}$ 
in exactly the same way as we constructed $f$ to interpolate $h$ in Theorem  \ref{ulb}; i.e.,  
\[ g(\alpha_i)=\widetilde{h}(\alpha_i), \ \ g^\prime(\alpha_i)=\widetilde{h}^\prime(\alpha_i), \ i=0,1,\ldots,k-1+\varepsilon, \]
where the interpolation is simple if and only if $(\varepsilon,i)=(1,0)$. It follows in the same way is in Theorem  \ref{ulb} that $g \in L_{\widetilde{h}}^{(n)}$,
implying immediately that $f\in L_{h}^{(n)}$.

In what follows in the proof we compute the bound from $f$ and show that it is greater than ULB$_{m}(W,h)$. 
Let $g(t)=\sum_{\ell=0}^{m} g_\ell P_\ell^{(n)}(t)$ be the Gegenbauer expansion of $g$. Note that $f_0=g_0$ and $f(1)=g(1)+\mu$.
We have 
\begin{eqnarray*}
\sum_{i=0}^{k-1+\varepsilon} \rho_i g(\alpha_i) &=& \sum_{i=0}^{k-1+\varepsilon} \rho_i \widetilde{h}(\alpha_i) \\
          &=& \sum_{i=0}^{k-1+\varepsilon} \rho_i h(\alpha_i) -\mu \cdot \sum_{i=0}^{k-1+\varepsilon} \rho_i P_j^{(n)}(\alpha_i) \\
          &=& \mathrm{ULB}_{m}(W,h)-\mu \cdot \sum_{i=0}^{k-1+\varepsilon} \rho_i P_j^{(n)}(\alpha_i). 
\end{eqnarray*}
Since $$\sum_{i=0}^{k-1+\varepsilon} \rho_i g(\alpha_i)=g_0-\frac{g(1)}{N_W}$$ by the formula \eqref{QR} for $g$ and
$$\sum_{i=0}^{k-1+\varepsilon} \rho_i P_j^{(n)}(\alpha_i)=Q_j(n,s)-\frac{1}{N_W} $$
by the definition of the test functions \eqref{test-functions}, we obtain
\[ g_0-\frac{g(1)}{N_W}=\mathrm{ULB}_{m}(W,h)+\frac{\mu}{N_W}-\mu Q_j(n,s). \]
But the left-hand side is equal to $f_0-(f(1)-\mu)/N_W$, and so
\[ f_0-\frac{f(1)}{N_W}=\mathrm{ULB}_{m}(W,h)-\mu Q_j(n,s)>\mathrm{ULB}_{m}(W,h), \] 
showing that the bound from the polynomial $f$ is better than $\mathrm{ULB}_{m}(W,h)$. 
 \end{proof}

The better bound deduced from the polynomial $f$ from Theorem \ref{THM-TF}(b) can be computed numerically but it seems to not be the best that
can be obtained by higher degree polynomials. A technique, called skip 2-add 2, was developed for obtaining
bounds via higher degree polynomials in the equi-weighted case in \cite{BDHSS-mc}. It can be applied for weighted codes as well. 

The same test functions were defined and investigated in the case of upper bounds\footnote{In this case the universal bound is the 
Levenshtein bound.} for the maximal cardinality of spherical codes of
given dimension and maximal inner product in \cite{BDB96}. It follows from the investigation from \cite{BDB96} that the first two
nonzero test functions, $Q_{m+1}(n,s)$ and $Q_{m+2}(n,s)$, are always non-negative. This means that the assumption $Q_j(n,s)<0$
in Theorem  \ref{THM-TF} (b) is possible only for $j \geq m+3$; i.e., the degree of an improving polynomial (if any) will be at least $m+3$.

\section{Universal upper bounds for energy of weighted codes with given minimum distance}

\subsection{A general linear programming upper bound for weighted codes}

We assume now that $(C,W)$ is a weighted spherical code on $\mathbb{S}^{n-1}$ with $|C|=|W|=N$ and 
maximal inner product $s(C) \in [-1,1)$ (equivalently, minimum distance $d(C)=\sqrt{2(1-s(C)}$). In the sequel, 
we shall often omit $C$ in the notation of the maximal inner product. Then it is natural to consider upper bounds for
the quantity $\mathcal{U}^h(s,W)$ defined in \eqref{min-w}. As with Theorem \ref{t_lower} we first 
derive a general linear program. 

We now define the admissible set of polynomials for linear programming as 
\[ U_h^{(n,s)}:= \left\{g(t)=\sum_{i=0}^{\deg(g)} g_i P_i^{(n)}(t) : g(t) \geq h(t), t \in [-1,s], g_i \leq 0, i=1,\ldots,\deg(g)\right\}. \]
Then every polynomial from the set $U_h^{(n,s)}$ provides a upper bound for $\mathcal{U}^h(s,W)$ as shown in the next theorem.

\begin{theorem} \label{t_upper}
Let $s \in [-1,1)$ and $g \in U_h^{(n,s)}$. Then for every 
weighted code $(C,W)$ on $\mathbb{S}^{n-1}$ with tuple size $N$ and 
maximal inner product $s(C)=s$,
\[ E_h(C,W) \leq E_g(C,W) \leq g_0-g(1)\sum_{i=1}^{N} w_i^2 .\]

Consequently,
\begin{equation} \label{inf_uub}
\mathcal{U}^h(s,W) \leq \inf_{g \in U_h^{(n,s)}} \left(g_0-g(1)\sum_{i=1}^{N} w_i^2\right). \end{equation}
\end{theorem}

\begin{proof}
The first inequality follows obviously from $g \geq h$ in $[-1,s]$ since all inner products 
of $C$ belong to that interval; i.e.,
\[ E_h(C,W) = \sum_{i \neq j}w_iw_j h(x_i \cdot x_j) \leq \sum_{i \neq j}w_iw_j g(x_i \cdot x_j)=E_g(C,W). \]
For the second inequality, we estimate $E_g(C,W)$ from 
above as follows:
\begin{eqnarray*}
E_g(C,W) &=& \sum_{i,j}w_iw_j g(x_i \cdot x_j)-g(1)\sum_{i=1}^{N} w_i^2 \\
              &=& \sum_{\ell=0}^{\deg(g)} g_\ell \sum_{i,j} w_iw_j P_\ell^{(n)}(x_i \cdot x_j)
                      -g(1)\sum_{i=1}^{N} w_i^2 \\
              & \leq& g_0-g(1)\sum_{i=1}^{N} w_i^2.
\end{eqnarray*}
We used that  $\left(\sum_{i=1}^{N} w_i \right)^2=1$, $g_\ell \leq 0$ for $\ell \geq 1$, and $\sum_{i,j} w_iw_j P_\ell^{(n)}(x_i \cdot x_j)=M_\ell(C,W)  \geq 0$ for $\ell \geq 1$. 
Since $(C,W)$ was arbitrary with tuple size $|C|=N$ and maximal inner product $s$, \eqref{inf_uub} follows. 
\end{proof}

By analogy with the ULB case we consider the infimum in \eqref{inf_uub} over the 
class of polynomials $U_h^{(n,s)} \cap \mathcal{P}_m$. Thus, we obtain 
the linear program
\begin{equation}\label{UUB_LP} 
\left\{ 
 \begin{array}{ll}
{\rm minimize} & \displaystyle g_0-g(1)\sum_{i=1}^{N} w_i^2 \\[8pt]
{\rm subject\  to} &  g \in U_h^{(n,s)} \cap \mathcal{P}_m.
 \end{array}  \right.
\end{equation}

\subsection{A construction of feasible polynomials}

We next construct polynomials in $ U_h^{(n,s)}$ to be used in the linear program 
\eqref{UUB_LP}. We  derive the necessary parameters and then follow the approach from the equi-weighted case \cite{BDHSS-dcc}. 

Given a weighted spherical code $(C,W)$ (in particular, a weight vector $W$) with maximal inner product $s(C)=s$ we can find $N_W$ and $m$ 
as before. However, the main role here will be played by another parameter denoted by $N_q$. 

Referring to the partition $[-1,1)=\cup_{j=1}^{\infty} I_j$ (see Section 2.2), we find the unique positive integer $q$ such that 
\[ s \in I_q=\left[t_{\ell-1+\varepsilon}^{1,1-\varepsilon},t_{\ell}^{1,\varepsilon}\right), \ \ \] 
where $q=2\ell-1+\varepsilon$, $\varepsilon \in \{0,1\}$ indicates the parity of $q$. Then we proceed with introduction of the additional parameter
\[ N_q:=L_q(n,s)=\frac{f_q^{(n,s)}(1)}{f_0}, \] 
(recall that $f_{q}^{(n,s)}(t)$ is the Levenshtein polynomial corresponding to $s \in I_q$ (i.e., the polynomial, used by Levenshtein for obtaining the bound $L_q(n,s)$. Thus, $N_q$ determines upper bound parameters $(\alpha_i^{(n,s)},\rho_i^{(n,s)})_{i=0}^{\ell-1+\varepsilon}$ in the same way as $N_W$ does for the lower bound parameters $(\alpha_i,\rho_i)$ in Theorem \ref{ulb}. 

If $N_q<N$, the Levenshtein bound implies that there exist no $C \subset \mathbb{S}^{n-1}$ with $|C|=N$ and $s(C)=s$.  Therefore,
$N_q \geq N$ (with equality in the latter case following if and only if the equi-weighted code $C \subset \mathbb{S}^{n-1}$ with $|C|=N=|W|$ attains the Levenshtein bound; i.e., it is universally optimal) and $q \geq m$.
We also note that $N_q \geq N$ means that $s$ cannot be too small in $I_q$; in fact, it belongs to a subinterval of $I_q$ whose left end is the largest solution of $L_q(n,s)=N$ (see section 4.3 for explicit expressions of that left end in the cases $q=1$ and 2).

In the next theorem we construct feasible polynomials $g \in U_h^{(n,s)} \cap \mathcal{P}_q$, where $q$ is determined by $s$ as above, and 
compute the corresponding universal upper bound (UUB).

\begin{theorem} \label{uub} (UUB for weighted codes) Let $W$ be a weight vector with $N_W=1/\sum_{i=1}^N w_i^2$. 
Let $s \in I_q=[t_{\ell-1+\varepsilon}^{1,1-\varepsilon},t_{\ell}^{1,\varepsilon}]$, $q=2\ell-1+\varepsilon$, $\varepsilon \in \{0,1\}$, 
be such that $L_q(n,s)=N_q \geq N=|W|$ and $h$ satisfy $h^{(q)} \geq 0$ on $[-1,s]$. 
Then
\begin{equation} \label{uub-formula}
\mathcal{U}^h(s,W) \leq -\lambda^{*} f_0\left(1-\frac{N_q}{N_W}\right)+(g_T)_0 - \frac{g_T(1)}{N_W},
\end{equation}
where $T$ is the set of nodes in $[-1,s]$ specified below in \eqref{gt-def} and \eqref{T-def} and $g_T$ is the Hermite interpolating polynomial 
to $h$ in the nodes $T$; furthermore, $f_0$ is the zeroth Gegenbauer coefficient of the Levenshtein polynomial $f_q^{(n,s)}$ and 
$\lambda^{*}$ is a constant defined below in \eqref{l-star}.  
\end{theorem}

\begin{proof}
We consider the polynomials
\begin{equation}
\label{uub_pol}
g(t):=-\lambda f_{q}^{(n,s)}(t)+g_T(t)=\sum_{i=0}^{q} g_i P_i^{(n)}(t),
\end{equation}
where $\lambda \geq 0$ is a parameter (to be determined later) and 
\begin{equation} \label{gt-def}
g_T(t):=H_{h,T}(t)=\sum_{i=0}^{q-1} (g_T)_i P_i^{(n)}(t) 
\end{equation}
is the Hermite interpolating polynomial to the function $h(t)$ that agrees with $h(t)$ exactly in 
the points of a multiset $T$ that will be defined now. 

With $q$ in the role of $m$ from the lower bounds case, we denote $q=2\ell-1+\varepsilon$, $\varepsilon \in \{0,1\}$ giverning the parity of $q$, 
and by 
\[ \alpha_0^{(n,s)}<\alpha_1^{(n,s)}< \cdots < \alpha_{\ell-2+\varepsilon}^{(n,s)}<\alpha_{\ell-1+\varepsilon}^{(n,s)}=s \]
the roots of the Levenshtein polynomial $f_{q}^{(n,s)}(t)$. We define 
\begin{equation} \label{T-def}
T:= \left\{ \begin{array}{ll}
\{\alpha_0^{(n,s)},\alpha_0^{(n,s)},\alpha_1^{(n,s)},\alpha_1^{(n,s)},\ldots,\alpha_{\ell-2}^{(n,s)}, 
\alpha_{\ell-2}^{(n,s)}, \alpha_{\ell-1}^{(n,s)}=s\} & \mbox{if } q=2\ell-1 \\[8pt]
 \{\alpha_0^{(n,s)}=-1,\alpha_1^{(n,s)},\alpha_1^{(n,s)},\ldots,\alpha_{\ell-1}^{(n,s)}, 
\alpha_{\ell-1}^{(n,s)}, \alpha_{\ell}^{(n,s)}=s\} & \mbox{if } q=2\ell
\end{array} \right. 
\end{equation}
to be the multiset of these roots counted with their multiplicities. Note that $\deg(g_T)=q-1$.

In a sense, we use the polynomial  $f_{q}^{(n,s)}$ to ``adjust" $g_T$ in order to obtain a suitable $g \in U_h^{(n,s)}$. 
Indeed, we have $g_T(t) \geq h(t)$ for $t \in [-1,s]$ from the interpolation (as described above) and  
the error formula in the Hermite interpolation 
\[ h(t)-g_T(t)=\frac{h^{(q)}(\xi)}{q!}(t-\alpha_0^{(n,s)})^{2-\varepsilon}(t-s) 
\prod_{i=1}^{\ell-2+\varepsilon}(t-\alpha_i^{(n,s)})^2 , \quad \xi\in(-1,1), \]
(note we interpolate only the value of $h$ at $s$). Since $f_q^{(n,s)}(t) \leq 0$ for $t \in [-1,s]$ and $\lambda>0$, we conclude that 
\[ g(t)=-\lambda f_q^{(n,s)}(t)+g_T(t) \geq h(t), \ \ t \in [-1,s]. \] 

Moreover, since $f_i>0$, $i=1,\ldots,q$, in the Gegenbaeur expansion of the Levenshtein polynomial 
\[ f_{q}^{(n,s)}(t) =\sum_{i=0}^q f_i P_i^{(n)}(t) \] 
(see, e.g., \cite[Theorem 5.42]{Lev98}), it is clear that large enough $\lambda>0$ will ensure that 
\[ g_i=-\lambda f_i + (g_T)_i \leq 0, \ \ i=1,2,\ldots,q-1, \ g_q=-\lambda f_q<0 \]
showing that $g \in U_h^{(n,s)}$ for such $\lambda$. 
Finally, combining the interpolation conditions and the properties of the Levenshtein polynomial (i.e., $g_T(\alpha_i^{(n,s)})=h(\alpha_i^{(n,s)})$ and 
 $f_{q}^{(n,s)}(\alpha_i^{(n,s)})=0$), we see that $g(\alpha_i^{(n,s)})=h(\alpha_i^{(n,s)})$.
Note that $g^\prime(\alpha_i^{(n,s)})=h^\prime(\alpha_i^{(n,s)})$ follows also (except in the cases $\alpha_0^{(n,s)}=-1$ and $i=\ell-1+\varepsilon$) 
from \eqref{uub_pol} since $\left(f_q^{(n,s)}\right)^\prime (\alpha_i^{(n,s)})=0$ in all relevant cases.

Assume that we have chosen $\lambda=\lambda_0$ such that $g(t)=-\lambda_0 f_{q}^{(n,s)}(t)+g_T(t)$ belongs to $U_h^{(n,s)}$. Then the bound
provided via the so-chosen $g(t)$ in Theorem \ref{t_upper} can be calculated as follows:
\begin{eqnarray*}
\mathcal{U}^h(s,W) &\leq& g_0-g(1)\sum_{i=1}^{N} w_i^2 \\
              &=& -\lambda_0 f_0+(g_T)_0 - (-\lambda_0f_q^{(n,s)}(1)+g_T(1))\sum_{i=1}^{N} w_i^2 \\
              &=& -\lambda_0 f_0\left(1-L_q(n,s)\sum_{i=1}^{N} w_i^2\right)+(g_T)_0 - g_T(1)\sum_{i=1}^{N} w_i^2 \\
              &=&  -\lambda^{*} f_0\left(1-\frac{N_q}{N_W}\right)+(g_T)_0 - \frac{g_T(1)}{N_W}
\end{eqnarray*}
(note the presence of the coefficient $f_0$ of the Levenshtein polynomial $f_q^{(n,s)}$). We recall that 
\[ N_q =L_q(n,s)=f_q^{(n,s)}(1)/f_0 \geq N \geq N_W. \] 
The linear dependence on $\lambda_0$ of the bound means that $\lambda_0$ must be as small as possible. Summarizing the 
requirements for the best value of $\lambda_0$, which we denote by $\lambda^{*}$, we conclude that  
\begin{equation} \label{l-star}
\lambda^{*}:=\max \left\{ \frac{(g_T)_i}{f_i} : i \in I(g_T) \right\}, 
\end{equation}
where $I(g_T):=\{i \in \{1,2,\ldots,q-1\} : (g_T)_i>0\}$ is the set of the indices of the positive coefficients in the 
Gegenbauer expansion of $g_T$ (we choose $\lambda^{*}=0$ if the set $I(g_T)$ is empty). This completes the proof. 
\end{proof}

We emphasize the difference between the definitions of the nodes $(\alpha_i)_{i=0}^{k-1+\varepsilon}$ for 
the lower bounds (Section 2) and $(\alpha_i^{(n,s)})_{i=0}^{\ell-1+\varepsilon}$ for the upper bounds 
(this section). In the case of ULB we find the nodes via the equation \eqref{L}; i.e., they are defined via $N_W=1/\sum_{i=1}^N w_i^2$, 
while for UUB we derive them as roots of the polynomial $f_{q}^{(n,s)}(t)$ used for obtaining the Levenshtein bound $L_q(n,s)$ for the 
given $s$; i.e., they are defined via the number $s$.

\begin{remark} If $h$ is absolutely monotone, then $g_T$ is positive definite as in Theorem \ref{ulb}; i.e. $I(g_T)=\{1,2,\ldots,q-1\}$ and 
\begin{equation} \label{l-star-abs}
\lambda^{*}:=\max \left\{ \frac{(g_T)_1}{f_1}, \ldots, \frac{(g_T)_{q-1}}{f_{q-1}}
\right\}, \ q>1,
\end{equation}
and $\lambda^{*}=0$ for $q=1$. 
\end{remark}

The quadrature rule \eqref{QR} for $g_T$ with $N_q$ gives
\[ (g_T)_0-\frac{g_T(1)}{N_q}=\sum_{i=0}^{\ell-1+\varepsilon} \rho_i^{(n,s)}g_T(\alpha_i^{(n,s)})=
\sum_{i=0}^{\ell-1+\varepsilon} \rho_i^{(n,s)}h(\alpha_i^{(n,s)}), \]
where $g_T(\alpha_i^{(n,s)})=h(\alpha_i^{(n,s)})$, $i=0,1,\dots,\ell-1+\varepsilon$, from the interpolation. Using this, we can write the bound \eqref{uub-formula} as follows:
\[ \mathcal{U}^h(s,W) \leq \left(-\lambda^{*}f_0+\frac{g_T(1)}{N_q}\right)\left(1-\frac{N_q}{N_W}\right) +
\sum_{i=0}^{\ell-1+\varepsilon} \rho_i^{(n,s)}h(\alpha_i^{(n,s)}). \]
This formula involves the potential $h$ explicitly. 
%

\begin{remark}
We remark that in \cite[Theorem 3.2]{BDHSS-dcc} the requirement of absolute monotonicity of $h$ can be weakened in the same way as 
in Theorem \ref{uub}. 
\end{remark}

\begin{remark} \label{mq} Let $q^\prime$ be a positive integer in $[m,q]$. As proved in \cite[Section 4]{Lev92}, the function $L_{q^\prime}(n,s)$ can be defined in a larger interval whose left end coincides with the left end of $I_{q^\prime}$ and whose right end is $t_r^{0,\epsilon}>t_r^{1,\epsilon}$, where $q^\prime=2r-1+\epsilon$. The image of this extension is the interval $\left[D(n,q^\prime),+\infty \right)$. Our $s$ can belong to such an extended interval, where we
can find $N_{q^\prime}=L_{q^\prime}(n,s)$ and can proceed with upper bounds as in Theorem \ref{uub}. However, such bounds will involve less nodes and weights and will be weaker.  
\end{remark}

\subsection{Small degrees UUB}  \label{small-d-upper}

We present explicitly degree one and two UUB. They are valid for $N_q \in (D(n,q),D(n,q+1)]$ and $s$ in a certain interval 
whose left endpoint is determined from the Levenshtein bound $N_q=L_q(n,s) \geq N$.  

For $s \in [-1/(N-1),-1/n] \supset I_1$ to generate $N_1 \in (N,n+1]$ we consider the degree one UUB 
\eqref{uub-formula}, where the parameters are determined as follows: $q=1$, $L_1(n,s)=(s-1)/s=N_1$ as in section \ref{small-d-lower},
\[ g(t)=-\lambda f_1^{(n,s)}(t)+g_T(t)=-\lambda (t-s)+g_T(t) \]
is our linear programming polynomial \eqref{uub_pol}, 
and $\alpha_0^{(n,s)}=s$, $\rho_0^{(n,s)}=1/(1-s)$
are the corresponding Levenshtein's parameters. Then the polynomial $g_T(t)$ has degree $q-1=0$; i.e., it is a constant which is found from the 
interpolation equality $g_T(s)=h(s)$. Thus $\lambda^{*}=0$ and we find $f(t)=h(s)$ giving the 
(trivial) bound  
\[ \mathcal{U}^h(s,W) \leq \left(1-\sum_{i=1}^{N} w_i^2\right)h(s)=\left(1-\frac{1}{N_W}\right)h(s). \]
Indeed, this bound is straightforward upon estimating all terms in the energy sum $E_h(C,W)$ from above 
by the constant $h(s)$ and taking into account that $w_1+\cdots+w_N=1$.

For$s \in [(N-2n)/n(N-2),0] \supset I_2$ to give $N_2 \in (N,2n]$ we consider \eqref{uub-formula} for degree $q=2$. 
Let 
\[ N_2=L_2(n,s)=2n(1-s)/(1-ns). \] 
We construct the UUB polynomial
\[ g(t)=-\lambda f_2^{(n,s)}(t)+g_T(t)=-\lambda (t+1)(t-s)+g_T(t) \]
as described in Theorem \ref{uub}. The Levenshtein's parameters can be computed as in section \ref{small-d-lower}
but we need $\alpha_0^{(n,s)}=-1$ and $\alpha_1^{(n,s)}=s$ only. 
The degree one polynomial $g_T(t)$ is found from the interpolation set 
$T=\{-1,s\}$; whence $g_T(-1)=h(-1)$ and $g_T(s)=h(s)$. Thus,
\[ g_T(t)=\frac{h(s)-h(-1)}{1+s}\cdot t+\frac{h(s)+sh(-1)}{1+s}. \]
The coefficient $\lambda^{*}$ is chosen to make $g_1=0$ in the Gegenbauer expansion $g(t)=g_0+g_1P_1^{(n)}(t)+g_2P_2^{(n)}(t)$
 (there is only one element in the set from \eqref{l-star}). This gives 
\[ \lambda^{*}=\frac{h(s)-h(-1)}{1-s^2} \] and, therefore, 
\[ g(t)=-\frac{h(s)-h(-1)}{1-s^2} \cdot t^2+\frac{h(s)-s^2h(-1)}{1-s^2}. \]
We can use now directly Theorem \ref{t_upper}. Since
\[ g_0=\frac{(n-1)h(s)+(1-ns^2)h(-1)}{n(1-s^2)}, \ \ g(1)=h(-1), \]
we obtain 
\[ \mathcal{U}^h(s,W) \leq  g_0-\frac{g(1)}{N_W}= \frac{(n-1)h(s)+(1-ns^2)h(-1)}{n(1-s^2)}-\frac{h(-1)}{N_W}. \] 

The degree three UUB is quite complicated to be stated here. However, numerical calculations are feasible as in the case of ULB. 

\subsection{Examples}

We compute the UUB from Theorem \ref{uub} for the cases discussed in Section 2. Thereby, we obtain a strip where 
the weighted $h$-energy belongs for given $n$, $h$, and $W$ (and $s$ for the upper bounds). 

\begin{example} \label{32u}
We consider the case of the weighted code $(C_{32},W) \subset \mathbb{S}^2$ 
from Example \ref{32}. Since the maximal inner product of $C_{32}$ is 
\[ s=s(C_{32})=\sqrt{1+2/\sqrt{5}}/\sqrt{3} \approx 0.794654 \in I_9 = [0.765055\ldots, 0.80293\ldots] \]
(this is the constant $b$ from Example \ref{32}), we compute 
the UUB for $ \mathcal{U}^h(s,W)$ for the parameters $n=3$, the above $s$ implying that $q=9$, $N_9=L_9(3,s)= 34.4268...$, and for the 
Coulomb potential  $h(t)=1/\sqrt{2(1-t)}$.

The interpolation nodes are the roots of the Levenshtein polynomial $f_9^{(3,s)}$; i.e.,
\[ (\alpha_i^{(3,s)})_{i=0}^4 \approx (-0.9247,-0.6213,-0.1493,0.3703,0.7946). \]
The polynomial $g_T$ is positive definite (i.e., $(g_T)_i \geq 0$ for all $i$ and we use \eqref{l-star-abs}; 
this is, in fact, part of more general property) and
\[ \lambda^{*}=\frac{(g_T)_1}{f_1} \approx 7.47994 \]
via \eqref{l-star}. Thus our UUB polynomial $g$ has $g_1=0$ and $g_i<0$ for $2 \leq i \leq 9$. Finally, the UUB is
\[ \mathcal{U}^h(s,W) \leq 0.8234054. \]
We recall from Example \ref{32} that the actual Coulomb energy of $(C_{32},W)$ is $\approx 0.8050318$ and the ULB is $\approx 0.804786$.
\hfill $\Box$
\end{example}

\begin{example} \label{cube+biortho-u}
We compute the UUB for the parameters from the weighted codes from Example \ref{cube+biortho}.
For convenience and comparison, we recall the actual energy and the ULB from that example. The results are shown in Table \ref{tab:table4}.

\begin{table}[h!]
  \centering
  \caption{Approximate parameters and UUB for $(n,N,N_W,s,N_q)=(n,2n+2^n,N_W,s,N_q)$, $3 \leq n \leq 7$, $h(t)=(2(1-t))^{-(n-2)/2}$.}
  \label{tab:table4}
  \begin{tabular}{|c|c|c|c|c|c|c|c|c|} \hline
    $n$ & $N_W$ & $N$ & $s$ & $q$ & $N_q$ & $ULB$ & Energy of $(C_{qp},W)$ & $UUB$ \\
    \hline      
    $3$ & $13.95$ & $14$ & $1/\sqrt{3}$ & $5$ & $16.098$ & $0.7058$ & $0.7070$ & $0.7357$ \\ 
    \hline
    $4$ & $24$ & $24$ & $1/2$ & $5$ & $26$ & $0.5781$ & $0.5798$ & $0.5988$ \\ 
    \hline
    $5$ & $41.48$ & $42$ & $3/5$ & $6$ & $81.351$ & $0.4825$ & $0.4901$ & $0.708$ \\ 
    \hline
    $6$ & $71.44$ & $76$ & $2/3$ & $7$ & $289.561$ & $0.4074$ & $0.4314$ & $1.0421$ \\     
    \hline
    $7$ & $121.16$ & $142$ & $5/7$ & $8$ & $2228.146$ & $0.3462$ & $0.3993$ & $1.9464$ \\ 
    \hline
  \end{tabular}
\end{table}

\end{example}

%
%

\section{ULB and UUB for weighted spherical designs}\label{ULB_Design}

In the case when a weighted spherical code is also a weighted design, the ULB 
and UUB may be extended to wider class of potentials, which we now describe. 

\subsection{Properties of weighted spherical designs} 

We shall need the following equivalent definition of weighted $\tau$-designs. It was used as primary definition of cubatures in \cite{GS81}. 

We denote by ${\mathcal P}_{\tau,n}$ the space of polynomials in $n$ variables of degree at most $ \tau$ and $\mathbb{H}_\ell^n$ the subspace of homogeneous spherical harmonics of degree $\ell\leq \tau$, where $Z(n,\ell):=\dim (\mathbb{H}_\ell^n)$ and $\{ Y_{\ell,k}\}_{k=1}^{Z(n,\ell)}$ is an orthonormal basis of $\mathbb{H}_\ell^n$. We recall that a real-valued function on $\mathbb{S}^{n-1}$ is called a {\em spherical harmonic of degree $\ell$} if it is the restriction of a homogeneous polynomial $Y$ in $n$ variables of degree $\ell$ that is harmonic, i.e. for which $\triangle Y \equiv 0$. 

\begin{lemma}\label{SphericalQuadrature} A weighted spherical code $(C,W)\subset \mathbb{S}^{n-1}$, $C=(x_1,\ldots,x_N)$, is a weighted 
spherical $\tau$-design if and only if the following quadrature formula
\begin{equation}\label{SphQu} 
\int_{\mathbb{S}^{n-1}} p( y)\, d\sigma_n(y)=\sum_{i=1}^N w_i p(x_i),\end{equation}
holds for all polynomials in $n$ variables of total degree at most $\tau$. Here $\sigma_n$ is the  Lebesgue surface measure of $\mathbb{S}^{n-1}$ normalized so that $\sigma_n(\mathbb{S}^{n-1})=1$.
\end{lemma} 

\begin{proof} The proof is similar to the proof in the equi-weighted case. 
 Since the restriction  ${\mathcal P}_{\tau,n} \vert_{\mathbb{S}^{n-1}}$ is a direct sum of the orthogonal subspaces $\mathbb{H}_\ell^n$, $\ell=0,1,\dots,\tau$, it suffices to prove the lemma for $p\in \mathbb{H}_\ell^n$, $\ell\leq \tau$. 

Since $w_1+\dots+w_N=1$, the quadrature \eqref{SphQu} holds trivially for the constant polynomial; i.e. for $\ell=0$. Let us fix $1\leq \ell\leq \tau$. Then the left-hand side of \eqref{SphQu} vanishes as $p$ is harmonic and homogeneous of degree at least $1$ (the mean-value property holds). From the Addition formula \cite[Theorem 2]{M} (see also \cite[Formula $(5.1.14)$]{BHS}) we have
\[ \sum_{k=1}^{Z(n,\ell)} Y_{\ell,k}(x_i)Y_{\ell,k}(x_j)=Z(n,\ell)P_\ell^{(n)}(x_i\cdot x_j), \]
which implies that 
$$M_\ell(C,W)=\frac{1}{Z(n,\ell)}\sum_{k=1}^{Z(n,\ell)} \left( \sum_{i=1}^N w_i Y_{\ell,k}(x_i)\right)^2.$$
Thus, if $(C,W)$ is a weighted $\tau$-design, then $M_\ell(C,W)=0$ and the right-hand side of \eqref{SphQu} holds true for all $\{ Y_{\ell,k}\}_{k=1}^{Z(n,\ell)}$ and hence for all $p\in \mathbb{H}_\ell^n$. On the other hand, if the quadrature \eqref{SphQu} holds for all $p\in \mathbb{H}_\ell^n$, then it holds for all $Y_{\ell,k}$, which implies $M_\ell(C,W)=0$ and completes the proof.
\end{proof}

Next, we utilize the identity
\begin{equation} \label{main-identity}
f(1) \cdot \sum_{i=1}^N w_i^2 +\sum_{i \neq j} w_iw_j f(x_i \cdot x_j)=f_0 + \sum_{\ell=1}^{\tau} f_\ell M_\ell(C,W) 
\end{equation}
(holding for $\deg(f) \leq \tau$; used in the proof of Theorem \ref{ulb}) to see the following characterization of weighted spherical designs.

\begin{theorem}
\label{f0-designs}
A weighted spherical code $(C,W)$ on $\mathbb{S}^{n-1}$ is a weighted spherical $\tau$-design if and only if
\begin{equation} \label{inner-products-f}
\sum_{i \neq j} w_iw_j f(x_i \cdot x_j)=f_0 -f(1) \cdot \sum_{i=1}^N w_i^2
\end{equation}
holds for any real polynomial of degree at most $\tau$.
\end{theorem}

\begin{proof}
If $(C,W)$  is a weighted spherical $\tau$-design on $\mathbb{S}^{n-1}$, then $M_{\ell}(C,W)=0$ for $1 \leq \ell \leq \tau$
yields \eqref{inner-products-f} directly from \eqref{main-identity}. 

Conversely, assume that \eqref{inner-products-f} holds for any polynomial $f$ of degree at most $\tau$. Then with the polynomial 
$f(t)=\sum_{i=0}^\tau P_i^{(n)}(t)$ in \eqref{main-identity} and \eqref{inner-products-f} we obtain that
\[ \sum_{\ell=1}^\tau M_\ell(C,W)=0; \]
whence $M_{\ell}(C,W)=0$ for $1 \leq \ell \leq \tau$.
\end{proof}

The condition \eqref{inner-products-f} can be further specialized as we shall see in the next theorem, where we prove that the left-hand side 
of \eqref{inner-products-f} decomposes into $N$ equal parts (when $x=x_i \in C$).

\begin{theorem}
\label{f0-designs-by-points}
A weighted spherical code $(C,W)$ on $\mathbb{S}^{n-1}$ is a weighted spherical $\tau$-design if and only if
for any point $x \in \mathbb{S}^{n-1}$ the equality 
\begin{equation} \label{inner-products-by-points}
\sum_{j=1}^N w_j f(x \cdot x_j)=f_0 
\end{equation}
holds for any real polynomial $f$ of degree at most $\tau$ and every $i=1,2,\ldots,N$.
\end{theorem}

\begin{proof}
If $(C,W)$  is a weighted spherical $\tau$-design on $\mathbb{S}^{n-1}$, then $M_{\ell}(C,W)=0$ for $1 \leq \ell \leq \tau$. 
Let $f(t)=\sum_{k=0}^\tau f_k P_k^{(n)}(t)$ be the Gegenbauer expansion of $f$. Then for every $i=1,2,\dots,N$ we have that 
$p(x):=\sum_{j=1}^N w_j f(x \cdot x_j)\in \mathcal{P}_{\tau,n}$ and from Lemma \ref{SphericalQuadrature} and the Funk-Hecke formula
we derive that
\[\sum_{j=1}^N w_j f(x \cdot x_j)=\int_{\mathbb{S}^{n-1}}f(x \cdot y)\, d\sigma_n(y)=\int_{-1}^1 f(t)\, d\mu_n(t)=f_0.\]

Suppose now that the identity \eqref{inner-products-by-points} holds for all $i=1,2,\dots,N$. Applying it for $x_i \in C$, then 
multiplying by $w_i$ and adding for all $i$, we utilize $f(t)=P_\ell^{(n)}(t)$ in \eqref{inner-products-by-points} to derive that $M_\ell (C,W)=0$ for $1\leq \ell\leq \tau$. 
\end{proof}

Theorem \ref{f0-designs-by-points} is the weighted analog of the property of designs in the equi-weighted case expressed by \cite[Equation (1.10)]{FL95}. 
The identity \eqref{inner-products-by-points} from Theorem \ref{f0-designs-by-points} could be used with $x \in C$ to estimate the 
frequency and location of inner products of $C$ 
via the polynomial $f$. Such approach was used in \cite{BBD99} for 
obtaining bounds on the extreme (smallest and largest) inner products that imply some nonexistence results for (equi-weighted) spherical 
designs of odd strengths and odd cardinalities. 

Applications of Theorem \ref{f0-designs-by-points} for obtaining bounds for polarization of weighted spherical designs will
be considered in a future paper.

\subsection{ULB on energy for weighted spherical $\tau$-designs}

Assume that $(C,W)$ is a weighted spherical $\tau$-design on $\mathbb{S}^{n-1}$
such that $N_W \in \left(D(n,\tau),D(n,\tau+1)\right]$. Since all the moments $M_i(C,W)$,  $1 \leq i \leq \tau$, are equal to $0$, 
we do not need the conditions $f_i \geq 0$ for $i=1,2,\ldots,\tau$. The good set of polynomials will be
\[ L_h^{(n,\tau)}:= \left\{f(t)=\sum_{i=0}^{\tau} f_i P_i^{(n)}(t) : f(t) \leq h(t), t \in [-1,1)\right\}, \]
For the same reason, we can relax the conditions of the derivatives of the potential function $h$ to the requirement $h^{(\tau+1)} \geq 0$ to be used only in the Hermite error formula. 

With these observations, we restate Theorem \ref{ulb} as ULB for weighted designs. 

\begin{theorem}
\label{ulb-designs} (ULB for weighted designs) Let $(C,W)$ be a weighted spherical $\tau$-design, $\tau=2k-1+\varepsilon$, $\varepsilon \in \{0,1\}$, and 
$W$ be such that $N_W$ satisfies \eqref{int-sum} with $m=\tau$. If the potential function $h$ satisfies $h^{(\tau+1)}\geq 0$, then
\begin{equation} \label{ulb-formula-des}
E_h(C,W) \geq \sum_{i=0}^{k-1+\varepsilon} \rho_i h(\alpha_i), 
\end{equation}
where the parameters $(\alpha_i, \rho_i)_{i=0}^{k-1+\varepsilon}$ are the same as in Theorem \ref{ulb}. 
This bound cannot be improved by any polynomial $f \in L_h^{(n,\tau)}$.
\end{theorem}

\begin{proof} With parameters and interpolation as in Theorem \ref{ulb}, we see that $f \in L_h^{(n,\tau)}$ and 
\[ E_h(C,W) \geq f_0-\frac{f(1)}{N_W}=\sum_{i=0}^{k-1+\varepsilon} \rho_i h(\alpha_i). \]

\end{proof}

In the case of absolutely monotone potential $h$, the bounds \eqref{ulb-formula} and \eqref{ulb-formula-des} are both valid and coincide.

\subsection{UUB for weighted spherical $\tau$-designs}

As was the case for the ULB for weighted designs, the condition $g_\ell \leq 0$ for the linear 
programming polynomials for UUB is no longer necessary for $1 \leq \ell \leq \tau$
since $M_\ell(C,W)=0$ in \eqref{main-identity} for these $\ell$. 

We first reformulate Theorem \ref{t_upper}. We consider the set of good polynomials as
\[ V_h^{(n,s,\tau)}:= \left\{g(t)=\sum_{i=0}^{\deg(g)} g_i P_i^{(n)}(t) : \ g(t) \geq h(t), \ t \in [-1,s], \ \deg(g) \leq \tau \right\}. \]

\begin{theorem} \label{t_upper_designs}
If $g \in V_h^{(n,s,\tau)}$, then for every weighted 
spherical $\tau$-design $(C,W)$ on $\mathbb{S}^{n-1}$ with cardinality (size) $N$ and 
maximal inner product $s=s(C) \in [-1,1)$,
\[ E_h(C,W) \leq E_g(C,W) = g_0-g(1)\sum_{i=1}^N w_i^2 .\]
Consequently,
\begin{equation} \label{inf_uub_des}
E_h(C,W) \leq \inf_{g \in V_h^{(n,s,\tau)}} \left(g_0-g(1)\sum_{i=1}^N w_i^2\right). \end{equation}
\end{theorem}

\begin{proof}
Obvious from $g \geq h$ in $[-1,s]$ and \eqref{main-identity}.
\end{proof}

\begin{theorem} \label{uub-des-thm}
Let $h$ be such that $h^{(\tau)} \geq 0$ and $(C,W)$  be a weighted 
spherical $\tau$-design on $\mathbb{S}^{n-1}$ with cardinality (size) $N$ and maximal inner product $s=s(C) \in [-1,1)$. 
Assume that $N_q=L_q(n,s)$ satisfies \eqref{int-sum} with $q=\tau=2\ell-1+\varepsilon$ and $N_q$ instead of $N_W$. Then
\begin{equation} \label{uub-formula-des}
E_h(C,W) \leq \frac{(N_W-N_q)g_T(1)}{N_WN_q}+\sum_{i=0}^{\ell-1+\varepsilon} \rho_i^{(n,s)} h(\alpha_i^{(n,s)}), 
\end{equation}
where the set $T$ of interpolation nodes and the polynomial $g_T$ are as in Theorem \ref{uub}. 
\end{theorem}

\begin{proof}
We naturally consider an interpolant $g$ that stays above $h$ in $[-1,s]$ so we utilize the nodes
from the multiset $T$ from \eqref{T-def} but with $q$ replaced by $\tau$. This produces the polynomial
$g=g_T$, which can be also viewed as the polynomial from \eqref{uub_pol}
with $\lambda=0$ (i.e., without the correction via the Levenshtein polynomial
$f_\tau^{(n,s)}$). Since $\deg(g)=\deg(g_T)=\tau-1$, the verification of 
$g \geq h$ in $[-1,s]$ by the Hermitte interpolation error formula
\[ h(t)-g(t)=\frac{h^{(\tau)}(\xi)}{\tau!} (t-\alpha_0)^{2-\varepsilon}(t-s)
\prod_{i=1}^{\ell-2+\varepsilon}(t-\alpha_i)^2 , \quad \xi\in(-1,1), \]
requires $h^{(\tau)}(t) \geq 0$ only.  

Therefore, $g \in  V_h^{(n,s,\tau)}$ and Theorem \ref{t_upper_designs} implies
\[ E_h(C,W) \leq (g_T)_0 - g_T(1)\sum_{i=1}^N w_i^2, \]
which yields \eqref{uub-formula-des} by the quadrature rule \eqref{QR} applied for the polynomial $g$ with $N_q$ instead of $N_W$
and the interpolation equalities $g(\alpha_i^{(n,s)})=h(\alpha_i^{(n,s)})$.
\end{proof}

\begin{remark} \label{rem-des-restr}
The conditions of Theorem \ref{uub-des-thm} impose the requirements $D(n,m)<N<N_q \leq D(n,m+1)$ which restrict the applicability of the bound \eqref{uub-formula-des} to weighted spherical designs with relatively small cardinalities (see the examples below). 
\end{remark}

We illustrate the UUB \eqref{uub-formula-des} with continuations of our two examples. 

\begin{example} \label{32u-des} Theorem \ref{uub-des-thm} can be applied to Examples 2.5 and 4.6 with $n=3$, $\tau=9$, $N=32$, 
weights $W$ as in $(C_{32},W)$, and every potential with $h^{(9)} \geq 0$ (note that all three numbers $N_W<N<N_q$ belong to the ninth 
Delsarte-Goethals-Seidel interval $(D(3,9),D(3,10)]=(30,36]$). Using again the Coulomb potential  $h(t)=1/\sqrt{2(1-t)}$ for comparison and 
computing \eqref{uub-formula-des}, we obtain
\[ E_h(C,W) \leq 0.805816 \]
for such weighted spherical designs. This bound is very close to the actual energy of $(C_{32},W)$ and the corresponding 
ULB (recall that $E_h(C_{32},W) \approx 0.8050318$ and $\mathcal{E}^h(W) \geq 0.804786$). 
The interpolation nodes and the polynomial $g_T$ itself are the same as in Example \ref{32u}. \hfill $\Box$
\end{example}

\begin{example} \label{cube+biortho-u-des}
We apply  Theorem \ref{uub-des-thm} for parameters of weighted 5-designs as in Examples \ref{cube+biortho} and \ref{cube+biortho-u}
and for potentials $h$ with $h^{(5)}(t) \geq 0$. The numbers $N$ and $N_q$ belong to the interval $(D(n,5),D(n,6)]$ for $3 \leq n \leq 5$. 
We compute the UUB for the Coulomb potential again to obtain:
\[ E_h(C,W) \leq 0.70893 \]
for every weighted 5-design (with weights as $(C_{qp},W) \subset \mathbb{S}^2$) with 14 points on $\mathbb{S}^2$, 
\[ E_h(C,W) \leq 0.58111 \]
for any (equi-weighted) 5-design with 24 points on $\mathbb{S}^3$, and 
\[ E_h(C,W) \leq 0.500221 \]
for every weighted 5-design (with weights as $(C_{qp},W) \subset \mathbb{S}^4$) with 42 points on $\mathbb{S}^4$.
These values are very close to the actual energy and the corresponding ULBs (cf. Table 4). \hfill $\Box$
\end{example}

\noindent{\bf Acknowledgments.}  
The research of the second author was supported, in part, by Bulgarian NSF grant KP-06-N72/6-2023. 
The research of the third author is supported, in part, by the Lilly Endowment.
The research of the sixth author was supported, in part, by Contract BG-RRP-2.004-0008, Sofia University Marking Momentum for Innovation and Technological Transfer (SUMMIT), Work group 3.2.1. Numerical Analysis, Theory of Approximations and Their Applications (NATATA).

 The authors thank an anonymous referee for comments that significantly improved the exposition.

\end{document}